\documentclass[10pt,final]{siamltex}
\usepackage{amsmath}
\newtheorem{remark}{Remark}[section]
\newenvironment{revision}{}{}
\usepackage{url}

\usepackage{bm}
\usepackage{amssymb,version}
\usepackage{cases}
\usepackage{color}
\usepackage{float}
\usepackage{verbatim}
\usepackage{algorithm}      
\usepackage{algpseudocode}
\newtheorem{prop}{Proposition}

\usepackage{graphicx}
\usepackage{subfig}

\allowdisplaybreaks
\usepackage{chngcntr}
\counterwithout{figure}{section}
\renewcommand{\thefigure}{\arabic{figure}}

\newcommand{\normmm}[1]{{\left\vert\kern-0.25ex\left\vert\kern-0.25ex\left\vert #1
    \right\vert\kern-0.25ex\right\vert\kern-0.25ex\right\vert}}
\newcommand{\ttt}[1]{(-\frac{t^{n+1}}{T})}

\begin{document}
   \title{A unified framework of energy-stable splitting exponential  integrators for damped Hamiltonian systems
\thanks{This work is supported by the National Natural Science Foundation of China (Grant Nos. 12501580, 12671476, 12271302, 12131014, 12071471) and Shandong Provincial Natural Science Foundation for Outstanding Youth Scholar (Grant No. ZR2024JQ030).}}

\author{Lu Li\thanks{School of Mathematics (Zhuhai), Sun Yat-sen University, Zhuhai 519082, Guangdong, China. (Email: lilu86@mail.sysu.edu.cn, xuqq35@mail2.sysu.edu.cn)}
\and Xiaoli Li\thanks{Corresponding author. School of Mathematics and State Key Laboratory of Cryptography and Digital Economy Security, Shandong University, Jinan 250100, Shandong, China. Email: xiaolimath@sdu.edu.cn }
\and Zaijiu Shang\footnotemark[4] \footnotemark[5]
\and Quanquan Xu\footnotemark[2]}

\maketitle
\begingroup
\renewcommand{\thefootnote}{\fnsymbol{footnote}}
\footnotetext[4]{Center for Mathematics and Interdisciplinary Sciences at Fudan University, Shanghai 200433, China. Email: zaijiu@simis.cn}
\footnotetext[5]{Shanghai Institute for Mathematics and Interdisciplinary Sciences, Shanghai 200433, China.}
\endgroup

\begin{abstract}
In this work, we study long-time numerical integration of Hamiltonian systems subject to linear perturbations. We assume the energy-compatible linear-damping  condition, which ensures that the physical energy is nonincreasing along the linear damping subflow. Since the conservative and dissipative effects cannot always be merged into a single gradient-structured dissipation and classical energy-stable methods developed for gradient flows cannot directly extend to this setting, we propose a unified framework of two efficient and energy-stable splitting exponential integrators (SEI) to separately handle the dissipative and conservative parts: SEISAV (SEI based on the scalar auxiliary variable) and SEILM (SEI based on Lagrange multiplier). The SEISAV scheme composes the exact damping subflow with an exponential integrator based on the SAV update for the Hamiltonian subflow and requires solving only a scalar linear  equation at each time step. We prove a discrete modified-energy dissipation law for SEISAV that mirrors the continuous energy-dissipation mechanism. To enforce decay of the original energy rather than a modified one, we further develop SEILM by incorporating a Lagrange-multiplier formulation within the  splitting exponential framework, leading to solving only a scalar nonlinear equation at each time step. Numerical experiments on representative linearly damped Hamiltonian models confirm the designed second-order accuracy, energy stability, and competitive efficiency relative to state-of-the-art schemes, indicating that the proposed framework provides a simple and robust approach for simulating linearly perturbed Hamiltonian dynamics.
\end{abstract}



\section{Introduction}

We consider the following Hamiltonian-type system with a linear perturbation
\begin{equation}\label{eq:model}
  \dot z \;=\; S\,\nabla H(z)\;-\;D(t)\,z,
\end{equation}
where $H(z)$ is the energy function, $D(t)$ is a prescribed continuous, possibly time-dependent, linear damping matrix, and  the constant matrix $S$ is skew-symmetric or its symmetric part is negative semidefinite. If
\(S\) is  skew-symmetric, the unperturbed system is a
constant-structure Poisson system, including the canonical Hamiltonian
case. We use $\operatorname{sym}(A):=\tfrac12\big(A+A^\top\big)$ and write $A\succeq 0$ if  $\operatorname{sym}(A)$ is positive semidefinite, while $A\preceq 0$ means that $\operatorname{sym}(A)$ is negative semidefinite. 

We say that the linear damping is energy-compatible with $H$ if
\begin{equation}\label{eq:ECLD}
  \nabla H(z)^\top D(t)z\geq 0,
\end{equation}
for all relevant $(z,t)$. We refer to \eqref{eq:ECLD} as the
\emph{energy-compatible linear damping} (ECLD) condition.
Along the damping subflow $\dot z=-D(t)z$, condition
\eqref{eq:ECLD} gives
\begin{equation*}
  \frac{d}{dt}H\!\bigl(z(t)\bigr)
  =
  -\nabla H\!\bigl(z(t)\bigr)^\top D(t)z(t)
  \leq 0.
\end{equation*}
For the full system \eqref{eq:model}, the Hamiltonian subflow contributes
$\nabla H(z)^\top S\nabla H(z)$, which is zero if $S^\top=-S$ and
nonpositive if $S\preceq0$. Hence, condition \eqref{eq:ECLD} gives the
continuous energy-dissipation law in the considered setting. When the origin is a critical point of $H$, as is the case for the
numerical examples considered later, we define the energy-induced matrix
\begin{equation}\label{eq:PH}
  P_H(z)
  :=
  \int_0^1 \nabla^2 H(sz)\,ds.
\end{equation}
It sends state to an effort satisfying the exact identity $\nabla H(z)=P_H(z)\,z$. The matrix condition
\begin{equation}\label{eq:ECLD-matrix}
  \operatorname{sym}\!\bigl(P_H(z)D(t)\bigr)\succeq 0
\end{equation}
for all relevant $(z,t)$ provides a convenient sufficient
condition for \eqref{eq:ECLD}.  Two important classes of damped Hamiltonian systems are recovered as special cases.
(i) \emph{Momentum/velocity damping.} In canonical variables \(z=(q,p)^\top\) with
\(H(q,p)=\tfrac12 p^\top M p+V(q)\) and \(D(t)=\mathrm{diag}(0,\Gamma(t))\), 
condition \eqref{eq:ECLD-matrix} reduces to \(\operatorname{sym}\!\big(M\Gamma(t)\big)\succeq 0\).
This holds, for example, when \(M\) is symmetric positive definite and \(\Gamma(t)\) is a nonnegative scalar multiple of the identity, as considered in~\cite{liu2021dissipation}.
\noindent (ii) \emph{Quadratic energies.}
If \(H(z)=\tfrac12 z^\top K z\) with $K=K^\top$, then \(P_H(z)\equiv K\) and \eqref{eq:ECLD-matrix} becomes
\(\operatorname{sym}\!\big(KD(t)\big)\succeq 0\).
This condition is satisfied, for instance, when \(K\) is symmetric positive definite and \(D(t)\) is a nonnegative scalar multiple of the identity, or when both \(K\) and \(D(t)\) are diagonal  with nonnegative
products of corresponding diagonal entries, as considered in~\cite{moore2021exponential}. 

Models of the form \eqref{eq:model} arise naturally after first-order reformulations of second-order mechanical systems and from spatial semi-discretizations for many PDEs. For example, underdamped mechanical systems such as Duffing oscillators and oscillator chains, with velocity damping acting only on the momentum channel, fit case (i) \cite{liu2021dissipation}. The deterministic part of the underdamped Langevin dynamics also falls within this class \cite{leimkuhler2024contraction}. Certain semi-discrete formulations of Burgers' equations \cite{bhatt2016second} with linear velocity damping fall into case (ii), as they have a quadratic energy that satisfies \eqref{eq:ECLD-matrix} and thus ensures a continuous decay law.

In the state-of-the-art, two main toolchains have been particularly effective for constructing structure-preserving discretizations of \eqref{eq:model}. The first employs Lawson (exponential) transformations  \cite{lawson1967generalized} to absorb the linear damping and maps \eqref{eq:model} to a time-dependent Poisson system. Then structure-preserving discretizations such as the discrete-gradient  methods or symplectic Runge-Kutta methods were applied to the transformed systems to develop methods that can preserve the correct decay of suitable invariants in the original variables under mild assumptions; see, e.g., \cite{moore2021exponential,bhatt2017structure}. The second one is the operator splitting method \cite{mclachlan2002splitting}: the Hamiltonian and dissipative subflows are solved separately and then composed to preserve qualitative structure. Representative results include methods that preserve conformal symplecticity  \cite{bhatt2016second,bhatt2017structure, modin2011geometric,  sun2005structure}, conformal multi-symplecticity \cite{moore2009conformal,moore2013conformal,bhatt2019exponential,cai2017modelling} and the energy dissipation \cite{liu2021dissipation}. In particular, the recent work \cite{liu2021dissipation} combines the discrete gradient idea and the splitting technique to develop energy-dissipation preserving methods for the damped Hamiltonian system falling within case (i) . These methods demonstrate robust long-time stability; however, they are nonlinear and fully implicit.

Although energy-stable schemes for damped Hamiltonian systems \eqref{eq:model} exist, linearly implicit efficient methods remain scarce. Two such approaches were studied in \cite{uzunca2025linearly}, obtained by combining exponential transformations with Kahan’s method and with the polarized discrete gradient, respectively. However, these are restricted to polynomial Hamiltonians, and a rigorous proof of energy dissipation for the constructed scheme is not provided. 
This motivates linear and energy-stable methods for general non-polynomial conservative nonlinearities, with per-step computation reduced to a constant-coefficient linear solver, and even to a single one-dimensional algebraic equation. In this work, we construct such methods by combining operator splitting with the scalar auxiliary variable (SAV) \cite{shen2019new,shen2018scalar,li2022new} and the Lagrange multiplier (LM) methods \cite{cheng2020new,cheng2025unique}. These formulations
reduce the nonlinear update to a decoupled constant-coefficient linear
solve and a scalar equation, respectively, while the resulting schemes
satisfy unconditional modified- and original-energy stability,
respectively. Recent studies on SAV/LM methods focus on i) seeking alternative scalar transforms to remove the lower-bound assumption on the nonlinear potential \cite{yang2020roadmap, liu2020exponential};  ii) improving the computational efficiency and constructing high-order methods \cite{huang2020highly, gong2020arbitrarily,kemmochi2022scalar}; iii) finding relaxation/energy-optimized corrections to improve the fidelity to the original energy dissipation for the SAV type methods \cite{cheng2020new, jiang2022improving, huang2026weighted,huang2024computationally}. These SAV/LM-based schemes have primarily been developed for either conservative Hamiltonian systems or purely dissipative gradient flows. They do not directly address problems of the form \eqref{eq:model}, where Hamiltonian dynamics are coupled with an additional linear dissipative perturbation, within an operator-splitting framework. It is important to highlight that the term \(-D(t)z\) cannot always be incorporated into an operator acting on \(\nabla H(z)\) without imposing restrictive assumptions, as demonstrated by the damped defocusing generalized Korteweg-de Vries (gKdV) example discussed in this paper. Consequently, energy-stable methods designed for purely gradient flows do not directly generalize to equation \eqref{eq:model}, and even when a combined reformulation is feasible, it may not necessarily lead to improved efficiency \cite{liu2021dissipation}. 
This observation motivates the development of SAV/LM-based splitting integrators specifically designed for linearly perturbed Hamiltonian systems. The main contributions of this work are twofold:
\begin{itemize}
\item 
We develop a unified extended Hamiltonian-type  treatment of the SAV and LM
reformulations for the Hamiltonian subflow and establish the corresponding
discrete energy laws, building on the Hamiltonian-type treatment for SAV
introduced in~\cite{kemmochi2022scalar}.
\item We develop two efficient dissipation-preserving integrators for system \eqref{eq:model}:
\begin{itemize}
  \item \textbf{SEISAV} (Splitting Exponential Integrator based on SAV): a  Strang splitting scheme in which the damping subflow is integrated exactly and the Hamiltonian substep is obtained by an exponential integrator based on SAV. It involves solving only a single scalar linear equation, and we establish a decay estimate for a modified discrete energy at each time step.
  \item \textbf{SEILM} (Splitting Exponential Integrator based on Lagrange Multiplier): a splitting method in which the Hamiltonian substep is obtained using an exponential integrator based on the LM formulation, thereby enforcing dissipation of the original energy rather than a modified one. The resulting scheme requires solving a one-dimensional nonlinear equation for the scalar multiplier at each time step, while typically remaining low cost.
\end{itemize}
\end{itemize}

We organize the rest of the paper as follows. Section~\ref{sec:sav review} presents the reformulated SAV/LM framework. Section~\ref{sec:exponential Integrator}  introduces the exponential integrators based on the SAV/LM framework. Section~\ref{sec:schemes} introduces the two proposed splitting schemes. Section~\ref{sec:numerical example} presents numerical experiments demonstrating energy dissipation, convergence orders, and computational efficiency. In the last section, we  conclude the study. We also give some proof details in the 
Appendix.

\section{The SAV/LM reformulation}\label{sec:sav review}  
In this section, we formulate the augmented equations associated with
the SAV and LM formulation  within a Hamiltonian-type framework. This viewpoint for the SAV formulation was developed in
\cite{kemmochi2022scalar}. Here, we develop a parallel LM
formulation. This framework allows us to rederive the corresponding energy conservation and dissipation properties  of these methods for the following  semi-linear Hamiltonian-type system
\begin{equation}\label{eq:model-gradient flow}
  \dot z \;=\; S\,\nabla H(z)\;, \quad
  H(z) = \frac{1}{2} z^\top M z + V(z),
\end{equation}
where $M=M^\top\succeq0$, $S^\top=-S$ or \(S\preceq0\), and \(V(z)\) is the nonlinear potential. 

\subsection{The standard SAV approach}\label{standard sav}
Assume  \( V(z) \) is bounded from below so that \( V(z) + C > 0 \) for some constant \( C>0 \).  Define 
\[
r(t) = \sqrt{V(z) + C} > 0,
\]
and introduce an extended variable $\tilde{z}=[z^\top,r]^\top$.
Then,  equation \eqref{eq:model-gradient flow} coupled with 
\[
  \dot r = \frac{1}{2\sqrt{V(z)+C}}  \nabla V(z)^\top  \dot z
\]
can be rewritten as the extended system  
\begin{equation}\label{Extendsys-SAV}
\dot{\tilde{z}}(t) = \widetilde{S}(\tilde{z})\nabla \widetilde{H}(\tilde{z}),
\end{equation}
where
\begin{equation*}
\widetilde{S}(\tilde{z}) = \begin{bmatrix} 
S & \frac{S\nabla V(z)}{2\sqrt{V(z)+C}}\\
\frac{{\nabla V(z)}^\top S}{2\sqrt{V(z)+C}} & \frac{{\nabla V(z)}^\top S\nabla V(z)}{4(V(z)+C)}
\end{bmatrix}, \qquad
\widetilde{H}(\tilde{z})= \frac{1}{2}z^\top Mz + r^2-C.
\end{equation*}

\begin{prop}\label{SAV}
The extended system \eqref{Extendsys-SAV} inherits the conservation/dissipation structure of \eqref{eq:model-gradient flow}, i.e.,
\begin{equation*}
\frac{d}{dt}\widetilde{H}(\tilde z(t))
\left\{
\begin{aligned}
=0
&\quad \text{if } S^\top=-S,\\
\le 0
&\quad \text{if } S\preceq 0.
\end{aligned}
\right.
\end{equation*}

\end{prop}
\begin{proof}
Denote by 
\[
 a(z):=\frac{\nabla V(z)}{2\sqrt{V(z)+C}},\qquad
B(z):=\big[I\ \ a(z)\big].
\]
Then $\widetilde{S}(\tilde{z})$ can be rewritten as 
\[
\widetilde{S}(\tilde z):=
\begin{bmatrix} S & Sa(z)\\ a(z)^{\top}S & a(z)^{\top}Sa(z)\end{bmatrix}
= B(z)^{\top}S\,B(z).
\]

When $S$ is skew-symmetric,  $\widetilde{S}(\tilde{z})$ is also skew-symmetric. Therefore, the extended Hamiltonian is conserved along solutions, i.e.,
\[
\frac{d}{dt}\widetilde{H}(\tilde{z})=\nabla \widetilde{H}(\tilde{z})^\top\dot{\tilde{z}}(t)=\nabla \widetilde{H}(\tilde{z})^\top\widetilde{S}(\tilde{z})\nabla \widetilde{H}(\tilde{z})= 0.
\]

When $S\preceq 0$, we decompose $S$ by $S=S_s+S_a$, where $S_s$ and $S_a$ are the symmetric and skew-symmetric part of $S$ with the form

\[
S_s=\frac{S+S^\top}{2}, \qquad S_a=\frac{S-S^\top}{2}.
\]
Then, for the augmented system we have 
\[
\frac{d}{dt}\widetilde{H}(\tilde z)
=\nabla\widetilde{H}(\tilde z)^{\top}\,\widetilde{S}(\tilde z)\,\nabla\widetilde{H}(\tilde z)
=\nabla\widetilde{H}^{\top}\,\widetilde S_s\,\nabla\widetilde{H}\le 0,
\]
since the skew-symmetric part does not contribute and the symmetric part is negative semidefinite.
\end{proof}

Based on the extended system \eqref{Extendsys-SAV}  and Proposition \ref{SAV}, it is straightforward to construct   linearly implicit unconditionally energy-stable schemes by applying a discrete-gradient type method to the extended system, together with an explicit pre-computation  of $\widetilde{S}$. One such scheme is the following second-order Crank-Nicolson discretization, referred to as the SAV/CN scheme:
\begin{align*}
    \frac{z^{n+1} - z^{n}}{h}
        &= S \, \mu^{n+\frac{1}{2}}, 
        \\[10pt]
    \mu^{n+\frac{1}{2}}
        &= M \left( \frac{z^{n+1} + z^{n}}{2} \right)
        + \frac{r^{n+1} + r^{n}}{2\sqrt{V(\hat{z}^{\,n+\frac{1}{2}}) + C}}
        \nabla V\bigl(\hat{z}^{\,n+\frac{1}{2}}\bigr), 
        \\[10pt]
    r^{n+1} - r^{n}
        &= \frac{\nabla V\bigl(\hat{z}^{\,n+\frac{1}{2}}\bigr)^{\top}}
        {2\sqrt{V(\hat{z}^{\,n+\frac{1}{2}}) + C}}
        \left( z^{n+1} - z^{n} \right),
\end{align*}
where $\hat{z}^{n+\frac{1}{2}}$ is a locally second-order approximation to $z^{n+\frac{1}{2}}$, 
for instance,
\begin{equation}
\dfrac{\hat{z}^{\,n+\frac{1}{2}} - z^n}{h / 2}
= S\!\left(M z^{n} + \nabla V(z^{n})\right).
\label{eq:Zbar_update}
\end{equation}

\begin{prop}
\label{prop:SAV_CN_stability}
The discrete energy of SAV/CN method satisfies 
\begin{equation*}
\begin{split}
\widetilde{H}(z^{n+1}, r^{n+1}) - \widetilde{H}(z^{n}, r^{n})
& =h \left[{\nabla\widetilde{H}\Big(\frac{\tilde z^n+\tilde z^{n+1}}{2}\Big)}\right]^{\top} \widetilde{S}(\hat{z}^{\,n+\frac{1}{2}})\,\nabla\widetilde{H}\Big(\frac{\tilde z^n+\tilde z^{n+1}}{2}\Big)\\
& = h \big(\mu^{n+\frac{1}{2}}, S \mu^{n+\frac{1}{2}}\big), \\
\end{split}
\label{eq:SAV_energy_dissipation}
\end{equation*}
and the method is unconditionally energy-stable in the sense that 
\begin{equation*}
\widetilde{H}(z^{n+1}, r^{n+1})
\left\{
\begin{aligned}
=\widetilde{H}(z^{n}, r^{n})
&\quad \text{if } S^\top=-S,\\
\le \widetilde{H}(z^{n}, r^{n})
&\quad \text{if } S\preceq 0.
\end{aligned}
\right.
\end{equation*}

\end{prop}
With \(\hat z^{\,n+\frac12}\) precomputed, the matrix \(\widetilde S\) is frozen
and the SAV/CN update is precisely the implicit midpoint rule applied to the frozen
extended system. Since the extended energy function \(\widetilde H\) is quadratic, its AVF discrete gradient coincides with its gradient
evaluated at the midpoint. Therefore, the SAV/CN scheme is equivalent
to the AVF discrete-gradient method \cite{celledoni2012preserving} for the
frozen extended system, and the stated conservation/dissipation identity follows directly from the standard discrete-gradient argument. Throughout the remainder of the paper, we adopt the Crank-Nicolson time discretization for all SAV-related formulations unless stated otherwise.

An advantage of the SAV-based approach over other linearly implicit methods, including  Kahan-based methods \cite{celledoni2012geometric} and polarized gradient-based methods \cite{matsuo2001dissipative}, is that it requires solving linear systems with the same constant coefficient
matrix at all iterations,  allowing its factorization to be reused, whereas the other methods generally involve a variable coefficient
matrix at each iteration. However, the price of this efficiency is the loss of symmetry. When \( V(z) \) is not bounded from below, it can be reformulated as a bounded term 
by employing a splitting technique \cite{kemmochi2022scalar}. 
Another issue of SAV discretizations is that \(r^n\) is not
algebraically reset to \(\sqrt{V(z^n)+C}\) at each time level, so the energy
law using the SAV method generally concerns a modified energy rather than the original Hamiltonian.
Previous studies have shown that this auxiliary-variable defect can be controlled under standard SAV
error-analysis assumptions and further reduced by relaxed SAV
technique~\cite{jiang2022improving,huang2024computationally}.

\subsection{The Lagrange multiplier approach}\label{sec:sav  Lagrange multiplier}
The standard SAV reformulation  assumes that the nonlinear potential \(V(z)\) is bounded from below so that the auxiliary variable is well defined, and  the associated discrete energy is usually a modified energy, rather than the original physical energy. To overcome these issues, Cheng et al.~\cite{cheng2020new} proposed the LM reformulation and demonstrated that it  preserves (or dissipates) the original energy exactly with the expense of solving an additional algebraic nonlinear equation.   In this work, we further reinterpret the LM augmented system within a Hamiltonian-type framework and prove its corresponding energy conservation/dissipation properties.

We introduce an auxiliary function \(\eta(t)\), serving as a Lagrange multiplier, and extend the Hamiltonian system \eqref{eq:model-gradient flow} as
\begin{equation}\label{eq:eta-constraint}
\begin{split}
\dot z&= SMz+\eta(t)S\nabla V(z),\\
\dot V(z) &= \eta(t) \nabla V(z)^T \dot z.
\end{split}
\end{equation}
If \(\eta(t)=1\), then \eqref{eq:eta-constraint} reduces to the original
Hamiltonian system \eqref{eq:model-gradient flow}. Thus, the LM formulation
provides a consistent extended-system representation of the original Hamiltonian
system.

By defining an extended variable \(\tilde{z} = [z^T, V]^T\), the above extended equation  \eqref{eq:eta-constraint} can be rewritten as the following Hamiltonian-type system
\begin{equation}\label{Extendsys-LM-SAV}
\dot{\tilde{z}} = \widetilde{S}(z)\nabla H(\tilde{z}),
\end{equation}
where
\begin{equation*}\label{SAV lagrange equation}
\widetilde{S}(z) = \begin{bmatrix} 
S & \eta(t)S\nabla V(z)\\
 \eta(t)\nabla V(z)^\top S &\eta(t)^2\nabla V(z)^\top S\nabla V(z)
\end{bmatrix},  \quad
H(\tilde{z})= \frac{1}{2}z^\top Mz + V.
\end{equation*}
Based on the Hamiltonian-type reformulation \eqref{Extendsys-LM-SAV}, it is easy to obtain the following results about energy stability.

\begin{prop}\label{SAV Lagrange multiplier}
The extended system \eqref{Extendsys-LM-SAV} inherits the conservation/dissipation structure of \eqref{eq:model-gradient flow}, i.e., 
\begin{equation*}
\frac{d}{dt} H(\tilde z(t))
\left\{
\begin{aligned}
=0
&\quad \text{if } S^\top=-S,\\
\le 0
&\quad \text{if } S\preceq 0.
\end{aligned}
\right.
\end{equation*}

\end{prop}

Similar to the original SAV approach, there is the following efficient Crank-Nicolson discretization, referred to as the LM/CN scheme:
\begin{align*}
    \frac{z^{n+1} - z^{n}}{h}
        &= S \, \mu^{n+\frac12}, 
        \\[6pt]
    \mu^{n+\frac12}
        &= M (\frac{z^{n+1} + z^{n}}{2})
          + \eta^{n+\frac12} \, \nabla V\bigl(\hat{z}^{\,n+\frac12}\bigr), 
          \\[6pt]
    V(z^{n+1}) - V(z^{n})
        &= \eta^{n+\frac12} \,
          \nabla V\bigl(\hat{z}^{\,n+\frac12}\bigr)^{\top}(z^{n+1} - z^{n}), 
\end{align*}
where $\hat{z}^{n+\frac{1}{2}}$ is a local second-order  approximation of $z^{n+\frac{1}{2}}$. This scheme requires solving an extra nonlinear equation for the auxiliary variable $\eta^{n+\frac12}$ compared to the standard SAV, and its implementation details can be found in~\cite{cheng2020new}. 
Proposition \ref{SAV Lagrange multiplier}  is analogous to Proposition~\ref{SAV}. Similarly, we obtain the following result on the
discrete energy, which parallels
Proposition~\ref{prop:SAV_CN_stability}.

\begin{prop}
\label{prop:LM-SAV_CN_stability}
Provided that the LM/CN scheme is solvable, its discrete energy satisfies
\begin{equation*}
\begin{split}
H(\tilde z^{n+1}) - H(\tilde z^{n})
& =h \left[{\nabla H\Big(\frac{\tilde z^n+\tilde z^{n+1}}{2}\Big)}\right]^{\top} \widetilde{S}(\hat{z}^{\,n+\frac{1}{2}})\,\nabla H\Big(\frac{\tilde z^n+\tilde z^{n+1}}{2}\Big)\\
& = h \big(\mu^{n+\frac{1}{2}}, S \mu^{n+\frac{1}{2}}\big), \\
\end{split}
\label{eq:LM_energy_dissipation}
\end{equation*}
and the method is unconditionally energy-stable in the sense that 
\begin{equation*}
H(\tilde z^{n+1})
\left\{
\begin{aligned}
=H(\tilde z^{n})
&\quad \text{if } S^\top=-S,\\
\le H(\tilde z^{n})
&\quad \text{if } S\preceq 0.
\end{aligned}
\right.
\end{equation*}

\end{prop}

The LM/CN scheme can be interpreted as a constrained midpoint
discrete-gradient discretization of the frozen extended system, where the
structure matrix is fixed at \(\widetilde S(\hat z^{\,n+\frac12})\). Since the
extended Hamiltonian is quadratic in \(z\) and linear in the potential-value
component, the midpoint discrete gradient coincides with the AVF discrete
gradient, yielding the stated conclusion. 

\section{The exponential  integrators based on the SAV/LM framework} \label{sec:exponential Integrator}
For the semilinear problem \eqref{eq:model-gradient flow}, exponential integrators integrate the stiff linear part exactly through matrix exponentials based on the variation-of-constants formula and
the nonlinear part is  approximated,  allowing for  bigger step size meanwhile guaranteeing stable approximations \cite{hochbruck2010exponential,li2026global}. Let $A=hSM$. We introduce the standard $\varphi$-function 
\[
\varphi_1(A)
=
\int_0^1 e^{(1-\theta)A}\,d\theta,
\quad
A\varphi_1(A)=e^A-I,
\]
and use it to formulate the EISAV and EILM schemes.

\subsection{The standard EISAV method} \label{standard ESAV}
When combined with the SAV framework, we treat the stiff linear part exactly and apply the SAV/CN reformulation to the nonlinear potential. Although SAV-based exponential integrators have been developed for several gradient-flow problems \cite{wang2025energy,ju2022generalized}, their application to Hamiltonian systems, especially within a splitting framework for linearly perturbed damped Hamiltonian dynamics, appears to be much less explored.Here we introduce the EISAV scheme and prove its discrete energy stability. The resulting linearly implicit EISAV scheme is:
\begin{subequations}\label{eq:EISAV}
\begin{align}
    z^{n+1} &= e^{A} z^n + h \, \varphi_1(A) \, S \, f(\hat{z}^{n+\frac{1}{2}}, \bar{r}^{\,n+\frac{1}{2}}), \label{eq:EISAV_z_update} \\
    r^{n+1} &= r^n + \frac{1}{2} (g^n)^\top (z^{n+1} - z^n), \label{eq:EISAV_r_update}
\end{align}
\end{subequations}
where
\begin{align*}
    g^n= \frac{\nabla V(\hat{z}^{n+\frac{1}{2}})}{\sqrt{V(\hat{z}^{n+\frac{1}{2}}) + C}}, \
    \bar{r}^{\,n+\frac{1}{2}} = \frac{r^{n+1} + r^n}{2}, \
    f(\hat{z}^{n+\frac{1}{2}}, \bar{r}^{\,n+\frac{1}{2}}) = \bar{r}^{\,n+\frac{1}{2}}g^n,
\end{align*}
and  $\hat{z}^{n+\frac{1}{2}}$ is a locally second-order  approximation to $z^{n+\frac{1}{2}}$, as stated in \eqref{eq:Zbar_update}.

In what follows, we describe  the computational details for the EISAV method. Substituting $r^{n+1}$ from equation~\eqref{eq:EISAV_r_update} into equation~\eqref{eq:EISAV_z_update}, we obtain
\begin{align}
z^{n+1}
= e^{A} z^n
+ h \varphi_1(A) S g^n r^n
- \frac{h}{4} \varphi_1(A) S g^n (g^n)^\top z^n
+ \frac{h}{4} \varphi_1(A) S g^n (g^n)^\top z^{n+1}.
\label{eq:z_substituted}
\end{align}
Denoting by
\begin{equation*}
\tilde{b}
= e^{A} z^n
+ h \varphi_1(A) S g^n r^n
- \frac{h}{4} \varphi_1(A) S g^n (g^n)^\top z^n,
\label{eq:b_vector}
\end{equation*}
and taking the inner product of both sides of \eqref{eq:z_substituted} with $g^n$ yields
\begin{equation*}
(g^n, z^{n+1})
= (g^n, \tilde{b})
+ \frac{h}{4} (g^n)^\top \varphi_1(A) S g^n \, (g^n, z^{n+1}).
\label{eq:inner_product}
\end{equation*}
Solving the above scalar equation for $(g^n, z^{n+1})$, we obtain
\begin{equation}
(g^n, z^{n+1})
= \frac{(g^n, \tilde{b})}
{1 - \frac{h}{4} (g^n)^\top \varphi_1(A) S g^n } .
\label{eq:solution}
\end{equation}
Substituting \eqref{eq:solution} back into \eqref{eq:z_substituted}, the update of $z^{n+1}$ is given by
\begin{equation}
z^{n+1}
= \tilde{b}
+ \frac{h}{4} \varphi_1(A) S g^n \, (g^n, z^{n+1}).
\label{eq:z_solution}
\end{equation}
The scalar auxiliary variable is then updated as
\begin{equation}
r^{n+1}
= r^n
+ \frac{1}{2} (g^n, z^{n+1})
- \frac{1}{2} (g^n, z^{n}) .
\label{eq:r_final}
\end{equation}
We summarize the one-step updating procedure of EISAV in Algorithm \ref{alg:EISAV}.
\begin{algorithm}[H]
\caption{One-step update of the EISAV scheme}
\label{alg:EISAV}
\begin{algorithmic}[1]
\Statex \textbf{Given} $(z^{n},r^{n})$, $h$, and precomputed $\hat z^{\,n+\frac12}$.
\State Compute the scalar quantity $(g^n, z^{n+1})$ from \eqref{eq:solution}.
\State Update the solution variable $z^{n+1}$ via \eqref{eq:z_solution}.
\State Update the scalar auxiliary variable $r^{n+1}$ according to \eqref{eq:r_final}.
\Statex \textbf{Return} $(z^{n+1},r^{n+1})$.
\end{algorithmic}
\end{algorithm}

\begin{remark}[Solvability of EISAV]
\label{rem:EISAV-solvability}
\begin{revision}
Under the structural assumptions of $M$ and $S$ in equation \eqref{eq:model-gradient flow}, the matrix
\(\varphi_1(A)S\) has a negative semidefinite symmetric part; see
Lemma~\ref{app:EISAV-solvability} in the Appendix for more details. Therefore, the denominator in \eqref{eq:solution} is not smaller than one, and the scalar linear equation in the EISAV update is uniquely solvable.
\end{revision}
\end{remark}

\begin{prop}
\label{energy_conservation-ESAV}
The discrete energy of the linearly implicit EISAV scheme \eqref{eq:EISAV} satisfies
\begin{equation*}
\widetilde{H}(z^{n+1}, r^{n+1}) 
\begin{cases}
=~\widetilde{H}(z^{n}, r^{n}), & \text{if } S^\top=-S,\\[4pt]
\le~\widetilde{H}(z^{n}, r^{n}), & \text{if } S \preceq 0,
\end{cases}
\label{eq:ESAV_discrete_energy_law}
\end{equation*}
where the discrete energy function is defined by
\begin{equation*}
\widetilde{H}(z^n, r^n) = \frac{1}{2}(z^n)^{\top} M z^n + (r^n)^2 - C.
\label{eq:ESAV_discrete_energy}
\end{equation*}
\end{prop}
The proof of this proposition can be found in the Appendix.

\subsection{The EILM approach} \label{sec: LM exponential Integrator}

Using the same notation of \(A\) and \(\varphi_1\), we give the following EILM scheme:
\begin{equation}\label{eq:EI_LM_z}
\begin{split}
z^{n+1}
 &= e^{A} z^n
   + h\, \varphi_1(A)\, S\, \eta^{n+\frac12}\, \nabla V(\hat{z}^{\,n+\frac12}), \\
V(z^{n+1})
 &= V(z^n)
   +  \eta^{n+\frac12}\,
     \nabla V(\hat{z}^{\,n+\frac12})^\top
(z^{n+1}-z^n),
\end{split}
\end{equation}
where  $\hat{z}^{n+\frac{1}{2}}$ is a locally second-order approximation to $z^{n+\frac{1}{2}}$, and \(\eta^{n+\frac12}\) is a scalar parameter that can be obtained by solving a  scalar nonlinear equation.

Define the intermediate variables
\begin{equation}
\label{qandp}
p^{n+1} = e^{A} z^n,\quad
q^{n+1} = h\, \varphi_1(A)\, S\, \nabla V\!\left(\hat{z}^{\,n+\frac12}\right),
\end{equation}
then we can rewrite $z^{n+1}$ by the following relation
\begin{equation}
\label{EILMZn}
z^{n+1} = p^{n+1} + \eta^{n+\frac12} q^{n+1}.
\end{equation}
Substituting equation \eqref{EILMZn} into the second line of \eqref{eq:EI_LM_z} yields the following  nonlinear equation about $\eta^{n+\frac12}$:
\begin{equation}
\label{etainEILM}
 V(p^{n+1} + \eta^{n+\frac12} q^{n+1}) - V(z^n) 
= \eta^{n+\frac12} \nabla V(\hat{z}^{\,n+\frac12})^\top 
\bigl(p^{n+1} + \eta^{n+\frac12} q^{n+1} - z^n\bigr).   
\end{equation}
By solving \eqref{etainEILM}, we get $\eta^{n+\frac12}$. Then we can update $z^{n+1}$ from \eqref{EILMZn}. We summarize the one-step update of EILM scheme in Algorithm \ref{alg:EILM}.
\begin{algorithm}[H]
\caption{One-step update of the EILM scheme}
\label{alg:EILM}
\begin{algorithmic}[1]
\Statex \textbf{Given} $(z^{n},V(z^{n}))$, $h$, and precomputed $\hat z^{\,n+\frac12}$.
\State Compute $p^{n+1}$, $q^{n+1}$ from equation \eqref{qandp}.
\State Solve for $\eta^{n+\frac{1}{2}}$ via nonlinear equation \eqref{etainEILM}.
\State Update $z^{n+1}$ according to equation \eqref{EILMZn}.
\Statex \textbf{Return} $(z^{n+1},V(z^{n+1}))$.
\end{algorithmic}
\end{algorithm}

\begin{remark}[Solvability of EILM]
\label{rem:LM-solvability}
The EILM Hamiltonian substep involves solving the scalar nonlinear equation \eqref{etainEILM} for
\(\eta^{n+\frac12}\), and its unique solvability generally requires additional
assumptions. A local uniqueness result generally requires a nonzero derivative
of the scalar constraint with respect to the multiplier, together with suitable
smoothness, boundedness, and a sufficiently small time step. This is consistent
with the nondegeneracy condition and the small-step assumption used in recent
unique-solvability analyses of LM schemes for gradient flows~\cite{cheng2025unique}. When this derivative is small, the scalar equation may
become ill-conditioned, and Newton's method may be sensitive to the initial
guess or fail to converge. In such cases, one may reduce the time step or use a
modified LM strategy, such as temporarily setting \(\eta^{n+\frac12}=1\) \cite{cheng2025unique}. 
A complete global
unique-solvability theory for the present splitting exponential LM framework is
left for future work.
\end{remark}

\begin{prop}
\label{energy_conservation-LM-ESAV}
 Provided that the EILM scheme is solvable, its discrete energy satisfies
\begin{equation*}
H(z^{n+1}) 
\begin{cases}
= ~H(z^{n}),& \text{if } S^\top=-S,\\[4pt]
\le~H(z^{n}), & \text{if } S \preceq 0.
\end{cases}
\label{eq:discrete_energy_law}
\end{equation*}
where the discrete energy is 
\begin{equation*}
H(z^{n}) = \frac{1}{2}(z^n)^{\top} M z^n + V(z^n).
\label{eq:EILM_discrete_energy}
\end{equation*}
\end{prop}
The proof of this proposition can be found in the Appendix.

\section{The two splitting schemes based on EISAV and EILM}\label{sec:schemes}
For the dissipative problem \eqref{eq:model}, schemes developed for gradient flows are not directly applicable in general, since the conservative and dissipative effects cannot always be merged into a single gradient-structured dissipation. In this section, we propose two efficient and energy-stable splitting schemes based on exponential integrators. This approach separates the Hamiltonian subflow from the linear damping subflow, allowing the latter to be treated exactly and the former to be solved by a structure-preserving integrator. Specifically, we decompose \eqref{eq:model} into
\begin{equation*}\label{eq:split}
  \dot z = f^{[1]}(z):=S\nabla H(z) 
  \quad\text{and}\quad
  \dot z = f^{[2]}(z,t):=-D(t)z .
\end{equation*}
Throughout this section, we assume that $D(t)$ is either constant or mutually commuting in time, namely
\[
    D(t_1)D(t_2)=D(t_2)D(t_1),\quad \forall\, t_1,t_2 .
\]
Under this assumption, the linear subflow admits the exact evolution operator
\begin{equation*}\label{eq:lin-evo}
 z(t_{n+1})=\Psi^{(t_n,t_{n+1})}z(t_{n}),  \quad\text{with}\quad  \Psi^{(t_n,t_{n+1})} := \exp\!\Big( -\!\int_{t_n}^{t_{n+1}} D(\tau)\,d\tau \Big).
\end{equation*}
A Strang splitting step is given by the symmetric composition
\begin{equation*}\label{eq:strang}
  z^{n+1}
  \;=\; \Psi^{(t_n+\frac h2,\,t_{n+1})}\, 
         \Phi^{[H]}_{h}\!\Big(\Psi^{(t_n,\,t_n+\frac h2)}\,z^n \Big),
\end{equation*}
where $\Phi^{[H]}_{h}$ denotes a one-step integrator for the Hamiltonian subflow $f^{[1]}$. 
In what follows, we give two choices of \(\Phi^{[H]}_h\) based on the two exponential integrators introduced in Section \ref{sec:exponential Integrator}.

\subsection{Splitting EISAV method}\label{subsec:sesav}
We propose a formally second-order splitting scheme in which the Hamiltonian subflow is integrated by a linearly implicit EISAV step. The scalar auxiliary variable is initialized  by
\(r^0=\sqrt{V(z^0)+C}\),
and is thereafter treated as an independent numerical variable. Then we  obtain the following one-step splitting EISAV (SEISAV) scheme:
\begin{align}
 &(z^-,r^-)= \big(\Psi^{(t_n,\,t_n+\frac h2)}z^n,\ r^n\big), \label{eq:sesav:half-in-left}
\\[1mm]
 & z^{+} = e^{A} z^- + \frac h2\,\varphi_1(A)S(r^+ + r^-)\tilde g , \quad
   r^+ - r^- = \frac{1}{2}\,\tilde g^\top (z^+ - z^-),\label{eq:sesav:core-left}
   \\[1mm]
 & (z^{n+1},r^{n+1}) = \big(\Psi^{(t_n+\frac h2,\,t_{n+1})}z^{+} ,\ r^{+}\big),\label{eq:sesav:half-out-left}
\end{align}
where  $\tilde g:=\nabla V(\hat z^{\,n+\frac12})/\sqrt{V(\hat z^{\,n+\frac12})+C}$, with $\hat z^{\,n+\frac12}$ a locally second-order predictor of $z(t_{n+\frac12})$ in the Hamiltonian step. 

\subsection{Splitting EILM method}\label{subsec:selmsav}
\begin{revision}
We next construct a splitting scheme using the EILM formulation for the Hamiltonian subflow. In contrast to the SAV auxiliary variable,  the EILM
substep uses the actual potential difference to enforce the original-energy law. We recompute the potential from the updated state after each damping substep rather than freezing the old value, and the  one-step SEILM scheme is defined as follows:
\begin{align}
 & z^- = \Psi^{(t_n,\,t_n+\frac h2)}z^n,\label{eq:slmsav:half-in}
\\[1mm]
 &z^{+}= e^{A} z^- + h\,\varphi_1(A)\,S\tilde g\,\eta^{n+\frac{1}{2}},
   \quad
   V(z^+)-V(z^-) = \eta^{n+\frac{1}{2}}\,{\tilde g}^\top (z^{+} -z^-), \label{eq:selmsav:core}
   \\[1mm]
 & z^{n+1}=\Psi^{(t_n+\frac h2,\,t_{n+1})}z^{+}. \label{eq:slmsav:half-out}
\end{align}
\end{revision}Here $\tilde g:=\nabla V(\hat z^{\,n+\frac12})$, with $\hat z^{\,n+\frac12}$ is a locally second-order approximation of $z^{n+1/2}$. The scalar multiplier $\eta^{n+\frac12}$ is obtained from the scalar nonlinear equation in \eqref{eq:selmsav:core}. 

\subsection{Energy stability}\label{subsec:props2}
\par\medskip
\begin{lemma}\label{lemma-SND}
Let \(M =M^\top\succeq0\)  and 
$D(t)$ be a matrix such that $\operatorname{sym}(M D(t))\succeq 0$ for $ t \in [t_a,t_b]$.
Then, the solution of the following matrix differential equation 
\[
\dot\Gamma(t) = -D(t)\,\Gamma(t), \qquad \Gamma(t_a) = I,
\]
is invertible and satisfies 
\begin{equation*} \label{eq:gamma-M-gamma}
\begin{split}\begin{cases}
\Gamma(t)^{\top} M \Gamma(t) - M &\preceq 0,\\
M - \Gamma(t)^{-\top} M \Gamma(t)^{-1} & \preceq 0.
\end{cases}
\end{split}
\end{equation*}
\end{lemma}

\begin{proof}
Defining $P(t) = \Gamma(t)^{\top} M \Gamma(t)$ and differentiating it, we obtain 
\[
\begin{aligned}
\dot P(t)
&= \dot\Gamma(t)^{\top} M \Gamma(t)
  + \Gamma(t)^{\top} M \dot\Gamma(t) 
= -2\,\Gamma(t)^{\top} \operatorname{sym}(M D(t)) \Gamma(t)\preceq 0,
\end{aligned}
\]
in the Loewner order. This means that $P(t)$ is nonincreasing in $t$. Therefore, we have 
\[
P(t) \preceq P(t_a), \quad \forall\, t \in [t_a,t_b].
\]
Since $P(t_a) = \Gamma(t_a)^{\top} M \Gamma(t_a) = M$, we further obtain 
\[
P(t) = \Gamma(t)^{\top} M \Gamma(t)\preceq M.
\]
The matrix \(\Gamma(t)\) is invertible as the fundamental matrix of a linear
homogeneous system.
Since \(\Gamma(t)\) is invertible, we apply a congruence transformation with
\(\Gamma(t)^{-1}\) to the inequality \(\Gamma(t)^\top M\Gamma(t)\preceq M\) and get
\[
\Gamma(t)^{-\top}\big(\Gamma(t)^\top M\Gamma(t)\big)\Gamma(t)^{-1}
\preceq \Gamma(t)^{-\top}M\Gamma(t)^{-1}.
\]
The left-hand side equals \(M\), yielding \(M\preceq \Gamma(t)^{-T}M\Gamma(t)^{-1}\). 
\end{proof}

\begin{theorem}\label{theorem-E-disppition}
Assume that for all \(t\), the linear coefficient matrix \(D(t)\) satisfies $\operatorname{sym}\!\big(MD(t)\big)\succeq 0$, 
then the  SEISAV method preserves the dissipation
of the modified Hamiltonian:
\[
\widetilde{H}(z^{n+1}, r^{n+1}) - \widetilde{H}(z^{n}, r^{n})  \leq  0,
\]
where the modified energy has the form: 
\[
\widetilde{H}(z^{n}, r^{n})=\frac{1}{2}{(z^{n})}^\top Mz^{n}+({r^{n}})^2-C.
\]
\end{theorem}
\begin{proof}
Based on the  EISAV scheme, we have 
\begin{equation*}
    \begin{cases}
        z^{-} = \Psi^{(t_n,\, t_n+\tfrac{h}{2})}\, z^{n},\\
        r^{-} = r^{n},
    \end{cases}
    \quad
    \text{and}
    \quad
    \begin{cases}
        z^{n+1} = \Psi^{(t_n+\tfrac{h}{2},\, t_{n+1})}\, z^{+}, \\
        r^{n+1} = r^{+}.
    \end{cases}
\end{equation*}
For brevity, we introduce
\[
\Psi_1 :=  \Psi^{(t_n,\, t_n+\tfrac{h}{2})}, 
\quad 
\text{and}
\quad
\Psi_2 := \Psi^{(t_n+\tfrac{h}{2},\, t_{n+1})},
\]
then we have 
\[
z^{-} = \Psi_1 z^{n}, 
\quad 
\text{and}
\quad
z^{+} = \Psi_2^{-1} z^{n+1}.
\]
The EISAV Hamiltonian substep satisfies the modified-energy law
\[
\widetilde{H}(z^{+}, r^{+}) \le \widetilde{H}(z^{-}, r^{-}),
\]
with equality in the skew-symmetric case. Hence
\[
\widetilde{H}(\Psi_2^{-1} z^{n+1},\, r^{\,n+1})
\le
\widetilde{H}(\Psi_1 z^{n},\, r^{n}).
\]
This inequality gives
\[
\frac12 (z^{n+1})^{\top} 
\Psi_2^{-\top} M \Psi_2^{-1}
z^{n+1}
+ (r^{n+1})^{2} - C
\le
\frac12 (z^{n})^{\top} 
\Psi_1^{\top} M \Psi_1
z^{n}
+ (r^{n})^{2} - C .
\]
Subtracting \(\widetilde{H}(z^{n}, r^{n})\) from $\widetilde{H}(z^{n+1}, r^{n+1})$ and using the above relation, we obtain
\begin{equation*}
\begin{split}
&\widetilde{H}(z^{n+1}, r^{n+1}) 
- 
\widetilde{H}(z^{n}, r^{n})\\
&\le
\frac12 (z^{n+1})^{\top}
\bigl[
M - \Psi_2^{-\top} M \Psi_2^{-1}
\bigr]
z^{n+1}
+
\frac12 (z^{n})^{\top}
\bigl[
\Psi_1^{\top} M \Psi_1 - M
\bigr]
z^{n}\\
& \le 0.
\end{split}
\end{equation*}
 The last inequality follows from Lemma \ref{lemma-SND}.
\end{proof}

\begin{theorem}\label{theorem-SEILM-fullenergy}
\begin{revision}
Suppose the ECLD condition \eqref{eq:ECLD} holds for all relevant states and time levels along the damped subflow. We further assume that the scalar LM equation within the EILM Hamiltonian substep admits a solution. Then the SEILM method
satisfies
\[
  H(z^{n+1})\le H(z^n).
\]
\end{revision}
\end{theorem}

\begin{proof}
\begin{revision}
The first damping half-step gives
\[
  z^-=\Psi^{(t_n,\, t_n+\tfrac{h}{2})}z^n .
\]
Along the damping subflow \(\dot z=-D(t)z\), the ECLD condition implies
\[
  \frac{d}{dt}H(z(t))
  =
  -\nabla H\!\bigl(z(t)\bigr)^\top D(t)z(t)
  \le 0 .
\]
Therefore,
\[
  H(z^-)\le H(z^n).
\]
By Proposition~\ref{energy_conservation-LM-ESAV}, the EILM Hamiltonian
substep gives
\[
  H(z^+)\le H(z^-),
\]
with equality when \(S^\top=-S\). The second damping half-step gives
\[
  z^{n+1}=\Psi^{(t_n+\tfrac{h}{2},\, t_{n+1})}z^+ .
\]
Using the ECLD condition again along this damping subflow, we obtain
\[
  H(z^{n+1})\le H(z^+).
\]
Combining the last three inequalities yields \(H(z^{n+1})\le H(z^n)\).
\end{revision}
\end{proof}

\begin{remark}\label{rem:damping-assumptions}
The ECLD condition \eqref{eq:ECLD} ensures dissipation of the full
Hamiltonian along the damping subflow. Specifically, for any
\(t_a\le t_b\),
\(H\bigl(\Psi^{(t_a,t_b)}z(t_a)\bigr)
=
H\bigl(z(t_b)\bigr)
\le
H\bigl(z(t_a)\bigr)\).
In contrast, \(\operatorname{sym}(MD(t))\succeq0\) controls only the
quadratic part \(\frac12 z^\top Mz\). This weaker condition suffices for
the modified-energy stability of SEISAV, since \(r\) remains fixed
during each damping substep and hence only the quadratic part of
\(\widetilde H\) changes. Even under the ECLD condition, SEISAV satisfies
only a modified-energy law, because \(r^n\) is evolved independently
and generally does not satisfy
\((r^n)^2=V(z^n)+C\) exactly. For SEILM, original-energy decay requires
dissipation of the full Hamiltonian along the damping subflow, as
guaranteed by ECLD. If \(V\) is invariant under the damping subflow,
then the weaker condition
\(\operatorname{sym}(MD(t))\succeq0\) is also sufficient for the damping of
 the original energy, as shown in Examples 1 and 2 in the next section.

\end{remark}

\section{Numerical experiments} \label{sec:numerical example}
We evaluate the proposed two splitting schemes on benchmark problems and compare them with six state-of-the-art energy-stable integrators:
\begin{itemize}
\item \textbf{AVF}: the averaged vector field method in \cite{celledoni2012preserving}. 
\item \textbf{SAV}: the scalar auxiliary approach in \cite{shen2019new}, as stated in section \ref{standard sav}.
    \item \textbf{SAVF}: the splitting averaged vector field method in \cite{liu2021dissipation}.
    \item \textbf{SEAVF}: the splitting exponential averaged vector field method in \cite{liu2021dissipation}.
\item \textbf{SSAV}: the splitting scalar auxiliary variable approach, obtained by replacing the EISAV update for the Hamiltonian substep in \eqref{eq:sesav:core-left} with the standard  SAV discretization.
\item \textbf{SLM}: the splitting Lagrange multiplier approach, obtained by replacing the EILM update for the Hamiltonian substep in \eqref{eq:selmsav:core} with the Lagrange multiplier discretization.
\end{itemize}
AVF and SAV are applied only to the first two numerical examples, as only these problems can be rewritten in a gradient-type form required by both methods.

Boundary conditions are incorporated at the spatial semi-discretization stage
through the construction of the discrete finite-difference operators. Periodic
boundary conditions are imposed by periodic finite-difference matrices with
wrap-around stencils, while homogeneous Dirichlet boundary conditions are imposed by fixing the boundary values and evolving only the interior degrees of freedom. In both cases, the discrete operators are assembled so that the resulting finite-dimensional system retains the structural assumptions used in the proofs, including the structure property of \(S\), the symmetry and semi-definiteness of the discrete energy matrices, and the stated damping condition. The numerical error is measured by
\[
\mathcal{E} = \max_n \| z^n - z(t_n) \|_\infty.
\]
The energy error is defined by
\[
\mathcal{E}_H(t_n) = \bigl| H(z^n) - H(z(t_n)) \bigr|.
\]
The reference solutions are obtained using the SEAVF scheme, which has been shown to be effective for linearly perturbed damped systems, with a very small step size of $h_{\mathrm{ref}}=2^{-10}/100$ for Examples~1 and~2, and $h_{\mathrm{ref}}=2^{-12}/100$ for Example~3. For all SAV-related approaches, we take the parameter $C=10$. All nonlinear equations arising from fully implicit schemes are solved by fixed-point iteration with an absolute tolerance of $10^{-15}$. In the LM-related approaches, the scalar nonlinear equation for $\eta$ is solved by Newton's method with the same tolerance. We consider tolerances ranging from $10^{-15}$ to $10^{-6}$ and observe the same qualitative efficiency ranking in the reported tests. The matrix exponentials are approximated by Pad\'e formulas \cite{berland2007expint}, except that $e^A$ is computed using MATLAB's built-in function. All numerical experiments were conducted using MATLAB (R2026a) on a HUAWEI laptop (model BOHL-WXX9) equipped with an AMD Ryzen 7 4700U processor and 16 GB of RAM. An anonymized repository containing the implementation files and scripts for reproducing the reported experiments is available at \url{https://anonymous.4open.science/status/UF-ESEI-for-DHS-FAF1}.

\paragraph{\textbf{Example 1: Damped nonlinear Klein-Gordon equation}}
We consider the damped nonlinear Klein-Gordon equation \cite{cote2021long}  
\begin{equation}\label{eq:KG}
u_{tt}+\gamma u_t-\kappa u_{xx}+\omega^2 u+\alpha u^3=0,\ x\in[0,L],\;t>0,
\end{equation}
with periodic boundary conditions \(u(0,t)=u(L,t)\), where \(\kappa>0\), \(\alpha>0\), and \(\gamma>0\). Introducing the momentum variable \(v=u_t\) and defining \(z(x,t)=\bigl(u(x,t),v(x,t)\bigr)^{\!\top}\), equation \eqref{eq:KG} can be rewritten as  a damped Hamiltonian PDE.

Let \(N\in\mathbb{N}\) and introduce a uniform periodic grid on \([0,L]\) with spatial step size $\Delta x=L/N$. We approximate \(u_{xx}\) by the standard second-order central finite difference. Denote the semi-discrete unknowns by \(u_j(t)\approx u(x_j,t)\) and \(v_j(t)\approx u_t(x_j,t)\). To simplify notation, we continue to use \(z\) for the semi-discrete state. Let
\(q=(u_1,\ldots,u_N)^\top\), \(p=(v_1,\ldots,v_N)^\top\), and \(z=(q^\top,p^\top)^\top\in\mathbb{R}^{2N}\).
Then the resulting semi-discrete system can be written in the form of \eqref{eq:model}
with
\[
S=\begin{pmatrix}0&I\\-I&0\end{pmatrix},\qquad
D=\begin{pmatrix}0&0\\0&\gamma I\end{pmatrix},
\]
and the Hamiltonian
\[
H(z)=\frac12 p^\top p+\frac12 q^\top(\omega^2 I-\kappa R)q+\frac{\alpha}{4}\sum_{j=1}^N q_j^4,
\]
where \(R\in\mathbb{R}^{N\times N}\) denotes the periodic discrete Laplacian induced by the central-difference stencil. We set $L=2$, $N=20$, $\omega=\sqrt{0.2}$, $\kappa=0.04$, $\alpha=0.15$, $\gamma=0.1$, $T=50$ and consider the initial conditions
\[
u_j(0)=\frac{1}{\sqrt{5}}\sin\left(\frac{2\pi x_j}{L}\right),
\
v_j(0)=-\frac{3\pi}{5\sqrt{5}\,L}
\cos\left(\frac{2\pi x_j}{L}\right),\
x_j=j\Delta x,\ 1\le j\le N.
\]
\begin{figure}[H] 
  \centering
  \subfloat[Order plot\label{fig:orderA}]{
    \includegraphics[width=0.46\textwidth]{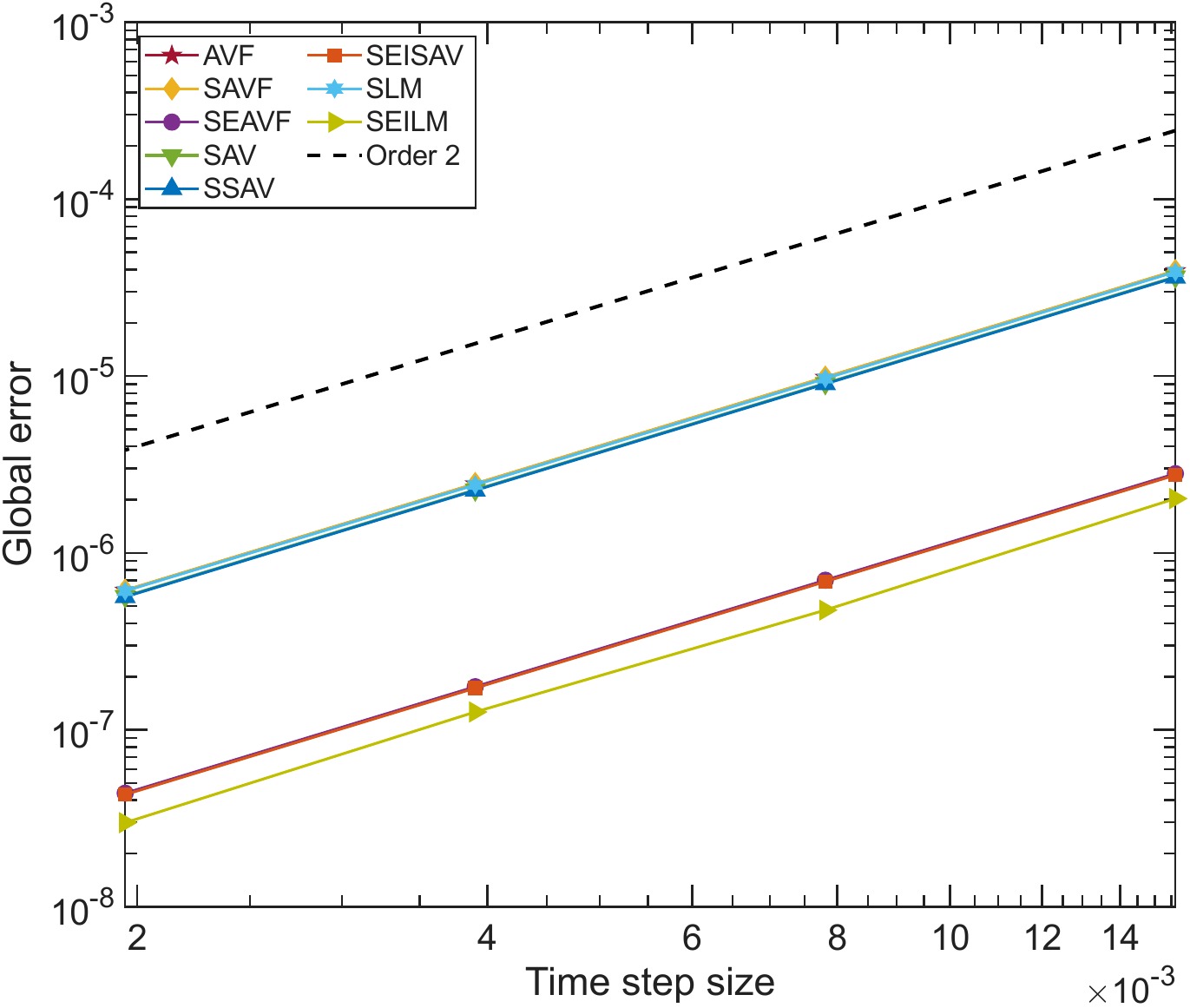}
  }\hfill
  \subfloat[Efficiency plot\label{fig:efficiencyA}]{
    \includegraphics[width=0.46\textwidth]{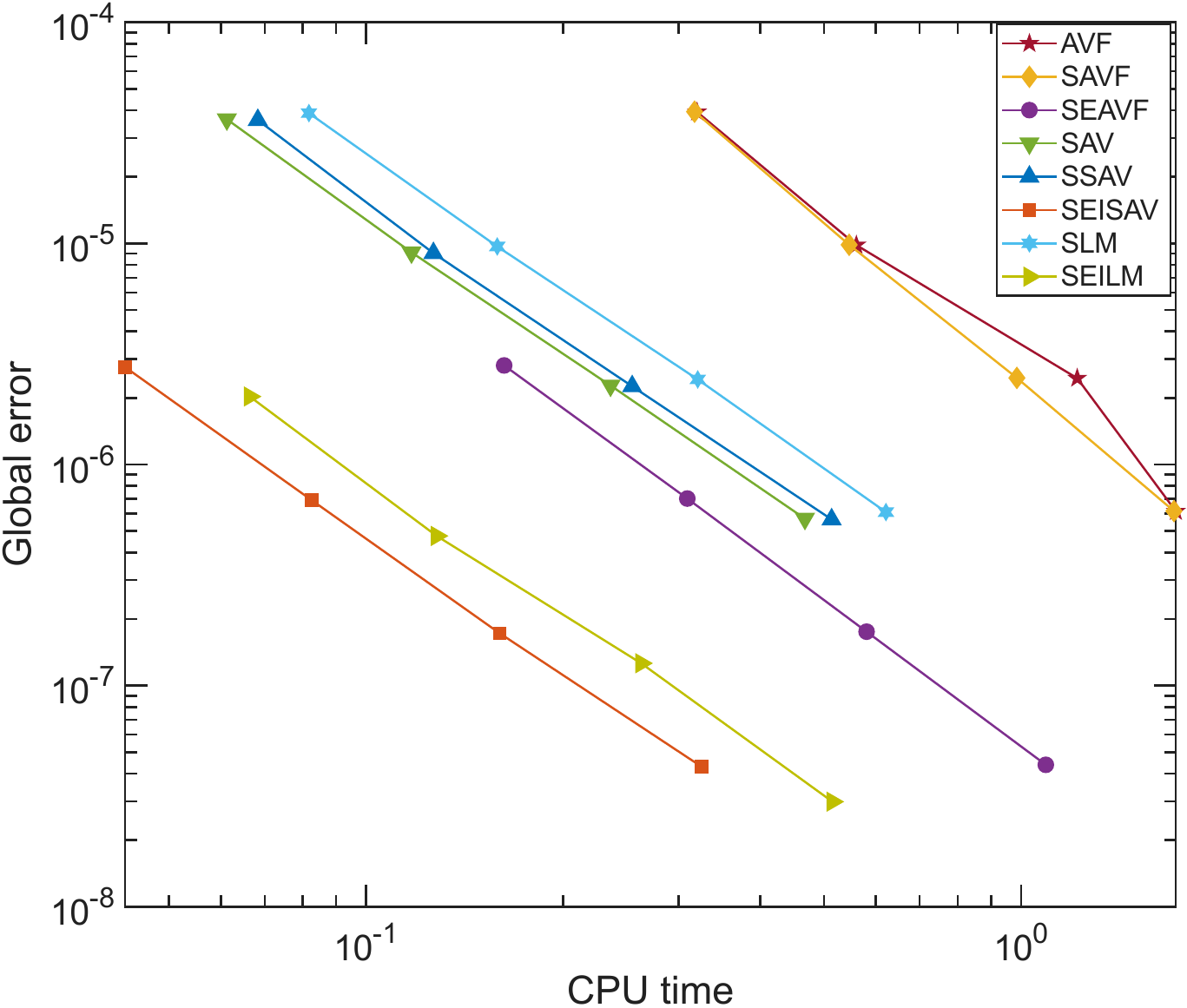}
  }

  \caption{Comparison of convergence orders (left) and computational efficiency (right) for the damped nonlinear Klein-Gordon equation.}
  \label{fig:figure1KG}
\end{figure}
 \begin{figure}[H] 
  \centering
 \subfloat[Energy evolution plot\label{fig:Energy decay A}]{
  \includegraphics[width=0.46\textwidth]{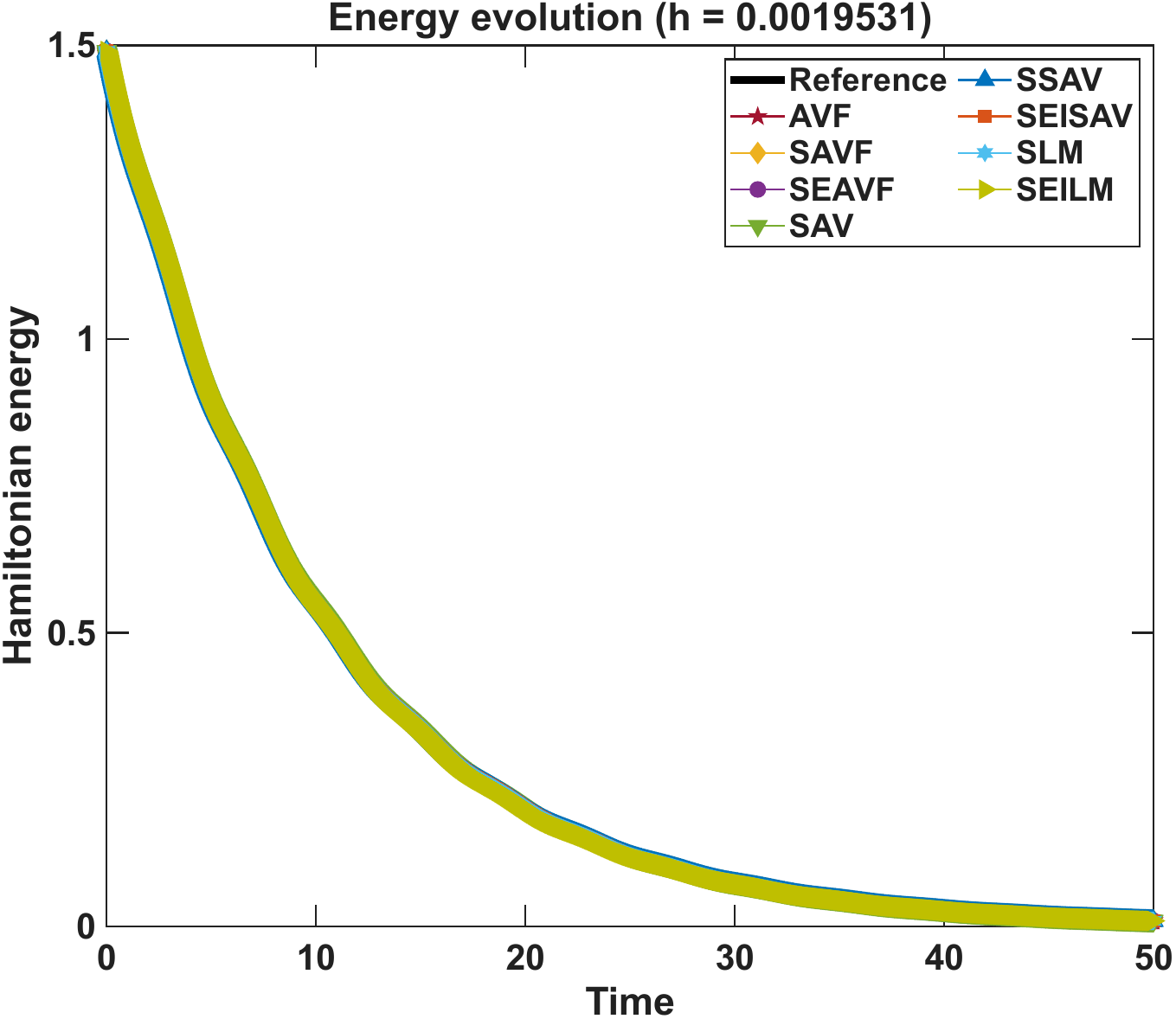}
  }\hfill
  \subfloat[Energy error plot\label{fig:Numerical energy error A}]{
\includegraphics[width=0.46\textwidth]{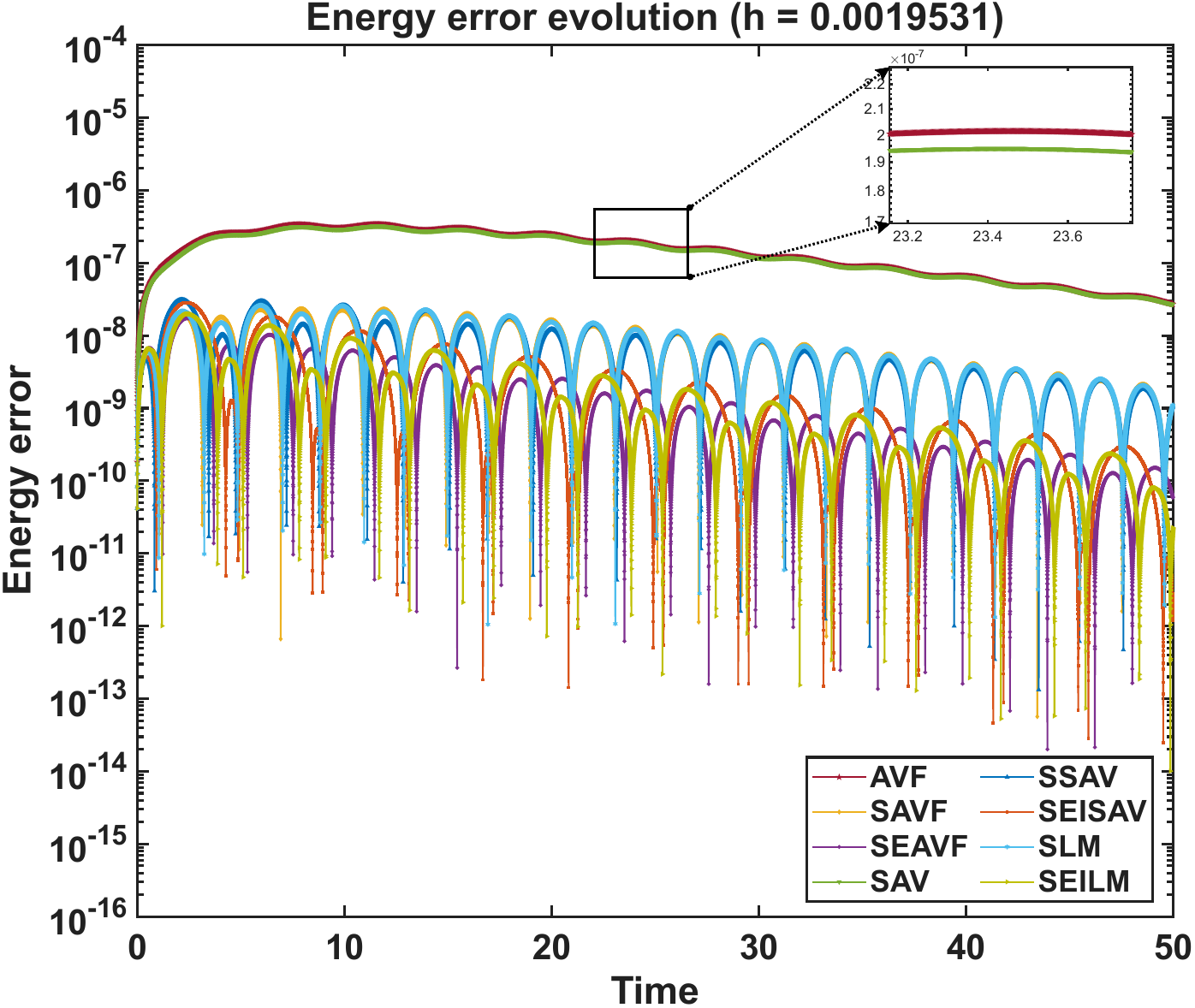}
  }
 \caption{Discrete energy evolution (left) and energy error (right) for the damped nonlinear Klein-Gordon equation.}
  \label{fig:figure2KG}
\end{figure}
Figure~\ref{fig:figure1KG} compares  the eight structure-preserving schemes with different time step sizes $h=1/2^{j+5}$ ($j=1,2,3,4$).  Figure~\ref{fig:orderA}  confirms that all methods exhibit a second-order convergence rate and that exponential integrators produce more accurate results than the other integrators. Figure \ref{fig:efficiencyA} reports the CPU time against accuracy, where we observe less computational cost for the SAV/LM-based methods than the other schemes; in particular, the proposed  SEISAV  and SEILM are substantially more efficient.  Figure~\ref{fig:Energy decay A} shows that, for all methods, the discrete energy \(H(z^{n})\) decays monotonically, consistent with the energy dissipation of the exact solution. Figure~\ref{fig:Numerical energy error A} shows that AVF exhibits a noticeable drift from the reference energy, while SAVF follows the reference energy much more closely. This difference highlights the benefit of the splitting strategy in capturing the energy structure of the damped Hamiltonian system. Moreover, from Figure \ref{fig:Numerical energy error A}, we observe that SEILM, SEISAV and SEAVF yield comparable energy errors, indicating that, for this test, the modified energy controlled by SEISAV closely tracks the original physical energy. 
To examine the effect of spatial resolution on computational efficiency,
we perform tests for Example~1 with \(N_x=80\) and \(N_x=160\).
Figure~3 shows that SEISAV and SEILM retain favorable accuracy-cost
trade-offs at the finer spatial resolutions, indicating that the observed
efficiency is not restricted to the coarse-grid test.
\begin{figure}[H] 
  \centering
  \subfloat[Efficiency plot $N_x=80$\label{fig:efficiency2}]{
    \includegraphics[width=0.46\textwidth]{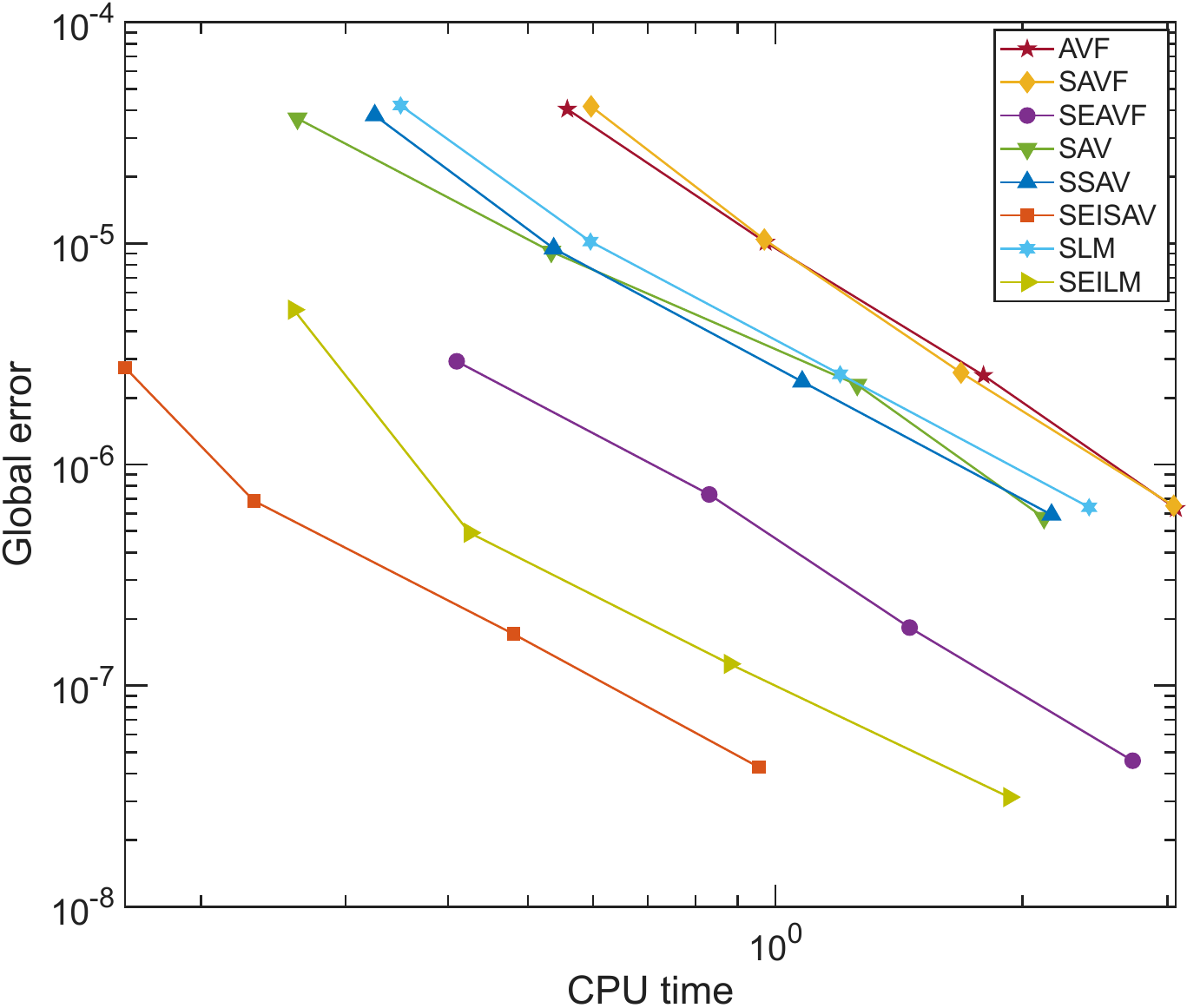}
  }\hfill
  \subfloat[Efficiency plot $N_x=160$\label{fig:efficiency3}]{
    \includegraphics[width=0.46\textwidth]{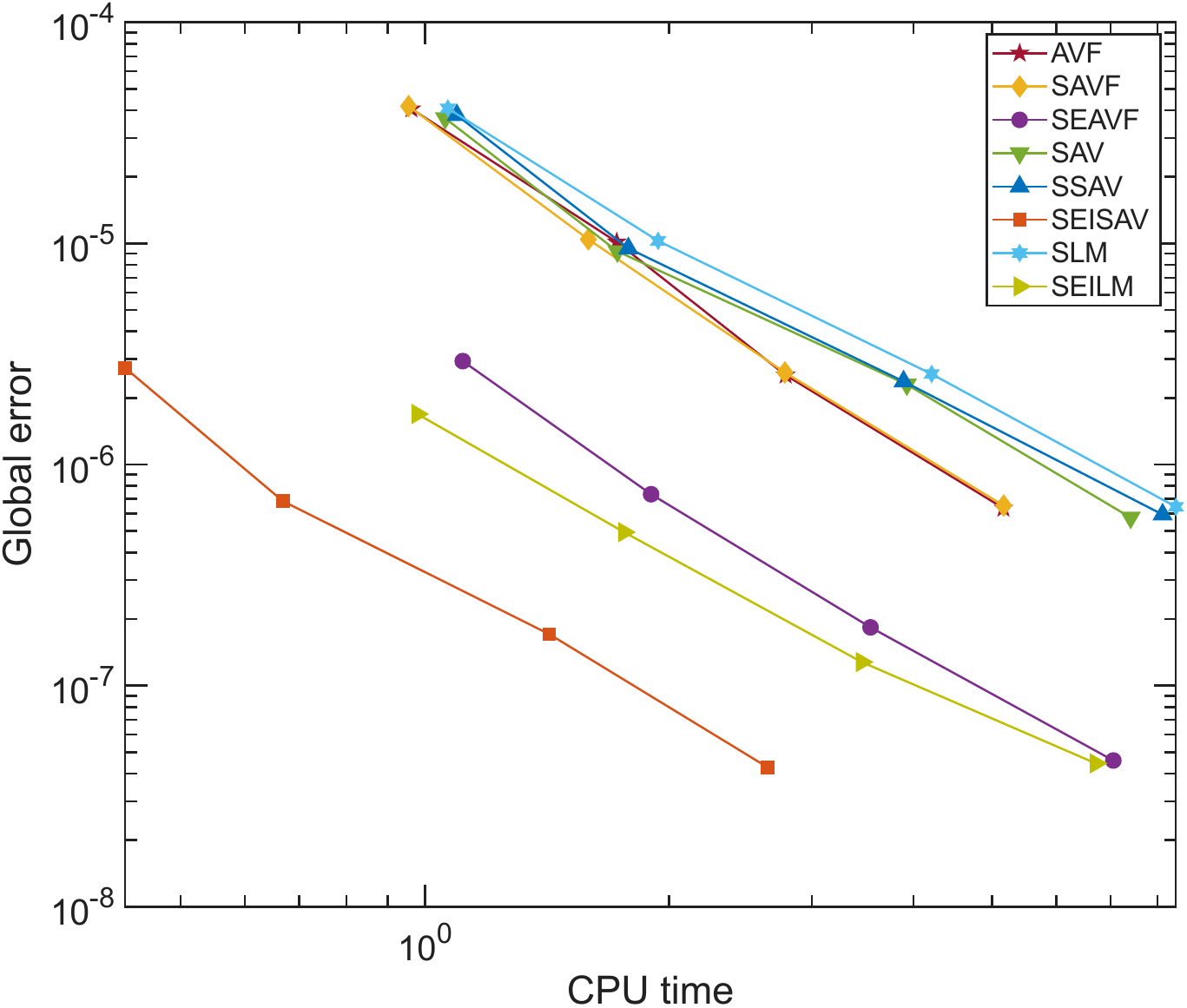}
  }

  \caption{Comparison of computational efficiency for the damped nonlinear Klein-Gordon equation with discretizations of $N_x = 80$ (left) and $N_x = 160$ (right).}
  \label{fig:figure3KG}
\end{figure}

\paragraph{\textbf{Example 2: Continuous FPU model with  damping terms}}
We study the dynamics of the one-dimensional damped Fermi-Pasta-Ulam (FPU) equation \cite{ullmann2000energy}
\begin{equation}\label{eq:afpu-pde}
u_{tt} - \beta u_{xxt} - u_{xx}(1+ \varepsilon u_x^{k}) + m^{2}u + \gamma u_t = 0,
\end{equation}
with homogeneous Dirichlet boundary conditions \(u(0,t)=u(L,t)=0\), where \(\varepsilon>0\), \(\beta\geq 0\), and \(\gamma\geq 0\).
Introducing the velocity variable \(v = u_t\), and defining  \(z(x,t)=\bigl(u(x,t),v(x,t)\bigr)^{\!\top}\), system \eqref{eq:afpu-pde} can be rewritten as a  damped Hamiltonian PDE
\begin{equation*}\label{eq:afpu-ham}
\partial_t z = \mathcal{S}\,\frac{\delta\mathcal{H}(z)}{\delta z} - \mathcal{D} z.
\end{equation*}
In the following, we set $k=2$, $L=2$, $N=20$, $\varepsilon=0.005$,
$\beta=0.01$, $\gamma=0.5$, $m=1$, $T=50$, and consider the following initial conditions 
\[
u_j(0)=0.1\sin(\pi x_j),
\
v_j(0)=0, \
x_j=j\Delta x,\ 1\le j\le N-1.
\]
We present the numerical results for the FPU system under two alternative choices of the operators \(\mathcal{S}\) and \(\mathcal{D}\).
\paragraph{\textbf{Example 2.1 -- $\mathcal{S}$ is a skew-adjoint operator}}
Introducing a uniform spatial grid, we obtain the  semi-discrete system
of the form \eqref{eq:model}, where \(z=(u^\top,v^\top)^\top\) and 

\[
S=\begin{pmatrix}0&I\\-I&0\end{pmatrix},\qquad
D=\begin{pmatrix}0&0\\0&\gamma I-\beta R\end{pmatrix}.
\]
The  Hamiltonian is
\begin{equation}\label{eq:afpu-diss-H}
H(z)=\frac12\,v^\top v+\frac12\,u^\top M u+V(u),
\end{equation}
where $M=m^2I-R$ and 
\(V(u)=\frac{\varepsilon}{\Delta x^{k+2}(k+1)(k+2)}
\sum_{j=0}^{N-1}\big(u_{j+1}-u_j\big)^{k+2}\).
Here \(R\in\mathbb{R}^{(N-1)\times(N-1)}\) is the standard second-order central-difference discretization of the one-dimensional Laplacian subject to homogeneous Dirichlet boundary conditions. 
\begin{figure}[H] 
  \centering
  \subfloat[Order plot\label{fig:orderB}]{
    \includegraphics[width=0.46\textwidth]{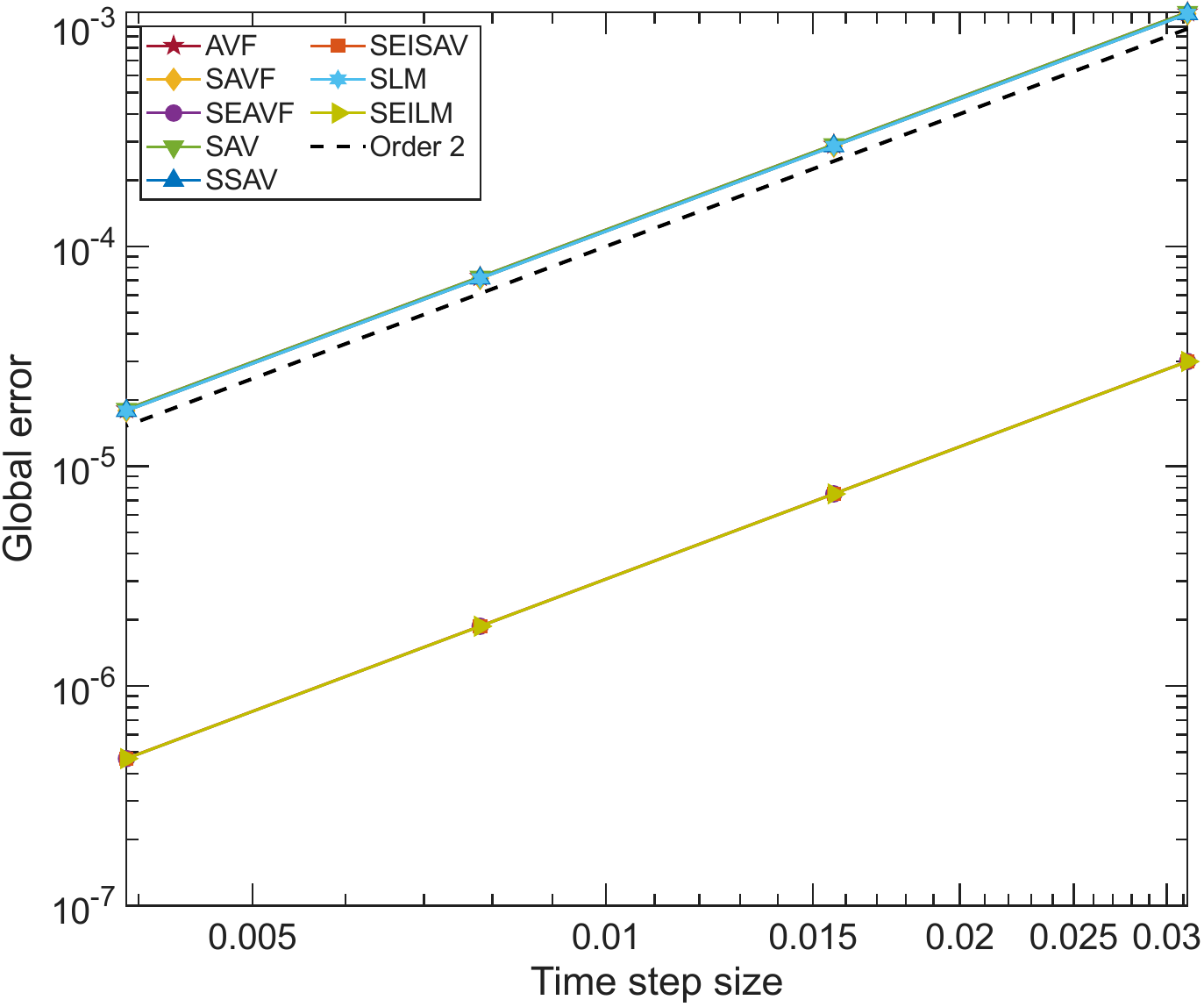}
  }\hfill
  \subfloat[Efficiency plot\label{fig:efficiencyB}]{
    \includegraphics[width=0.46\textwidth]{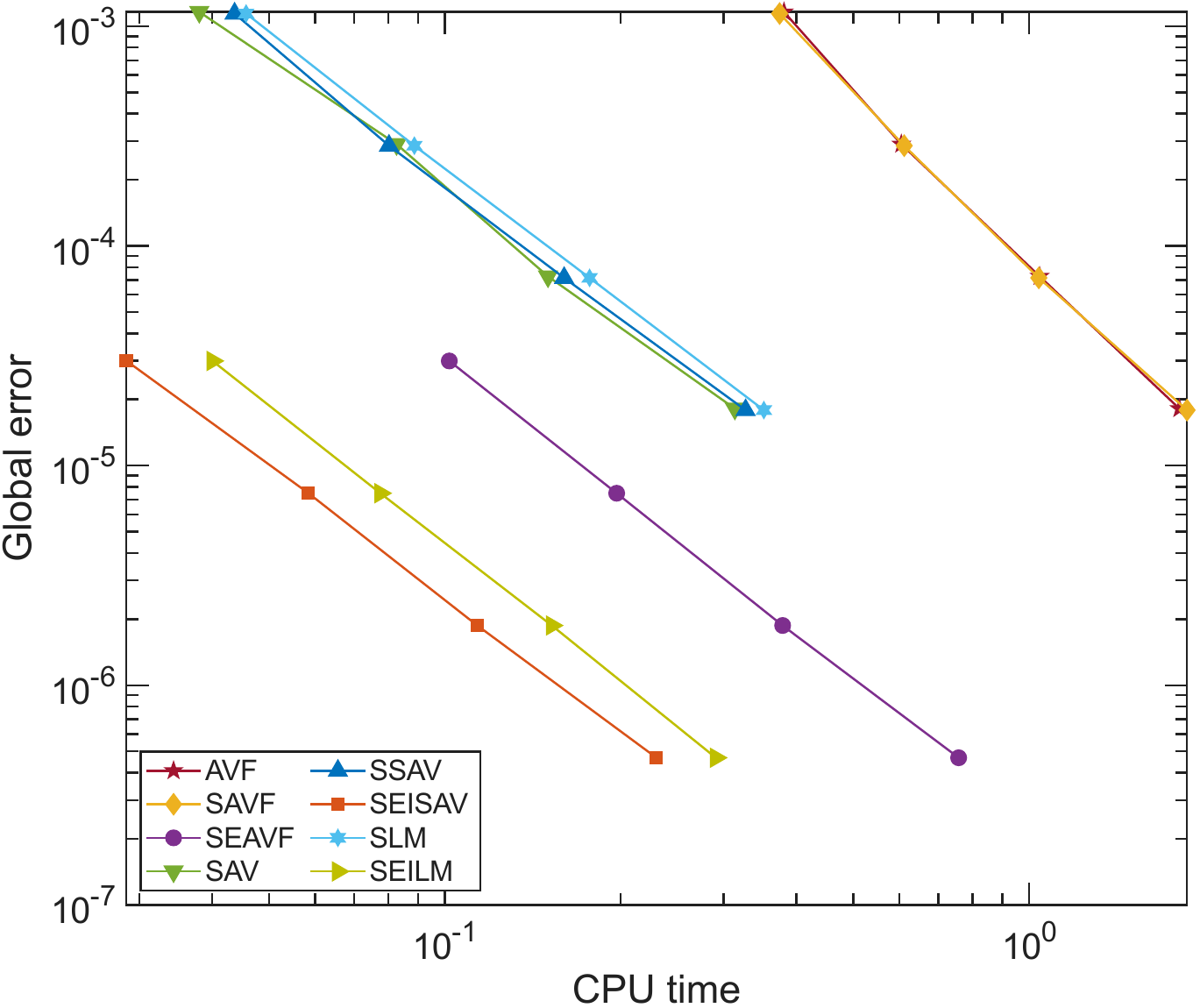}
  }
  \caption{Comparison of convergence orders and computational efficiency for the damped FPU equation in Example~2.1.}
  \label{fig:figure3FPU21}
\end{figure}

\begin{figure}
  \centering
  \subfloat[Energy evolution plot\label{fig:Energy decay B}]{
    \includegraphics[width=0.46\textwidth]{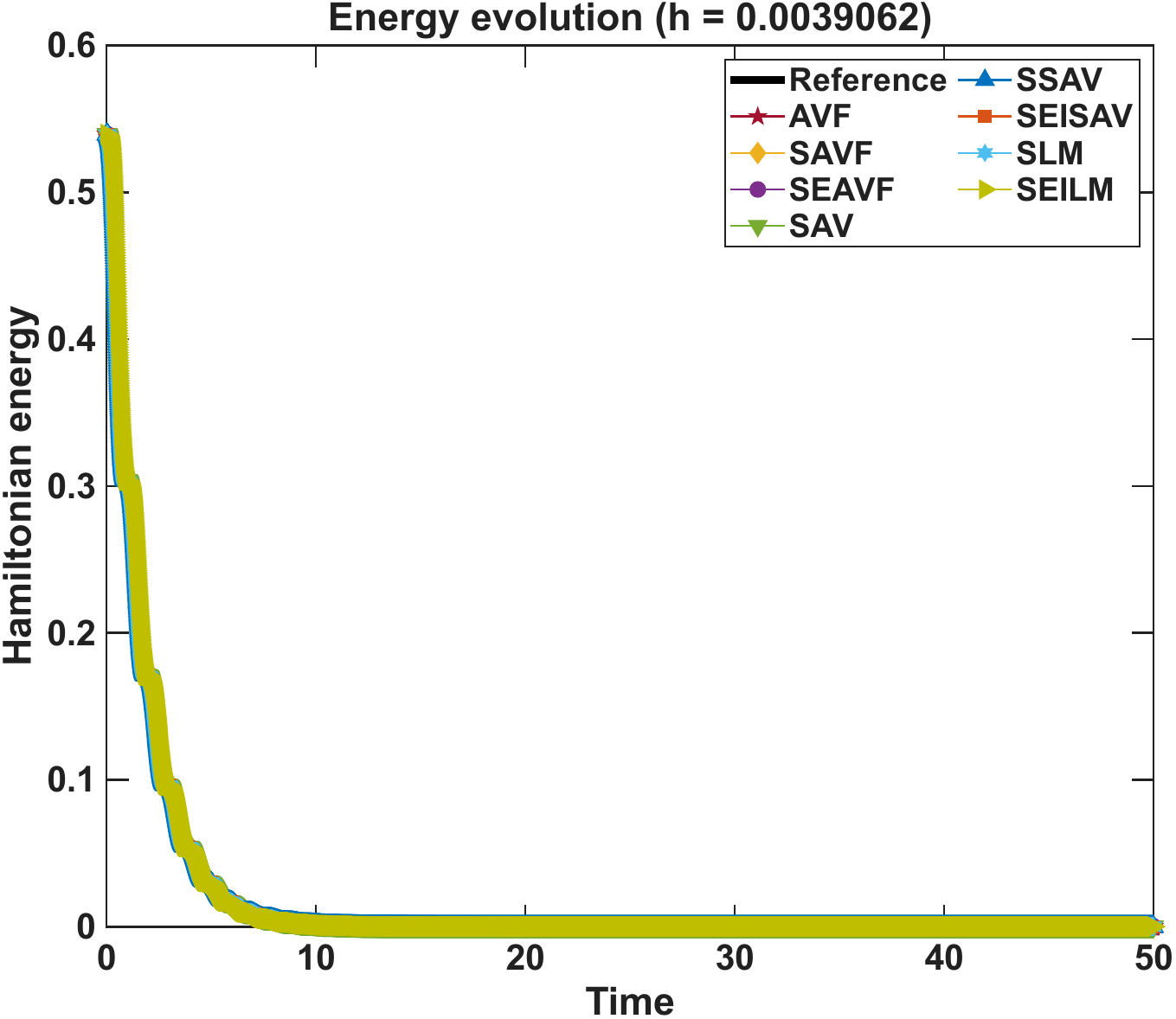}
  }\hfill
  \subfloat[Energy error plot\label{fig:Numerical solution B}]{
    \includegraphics[width=0.46\textwidth]{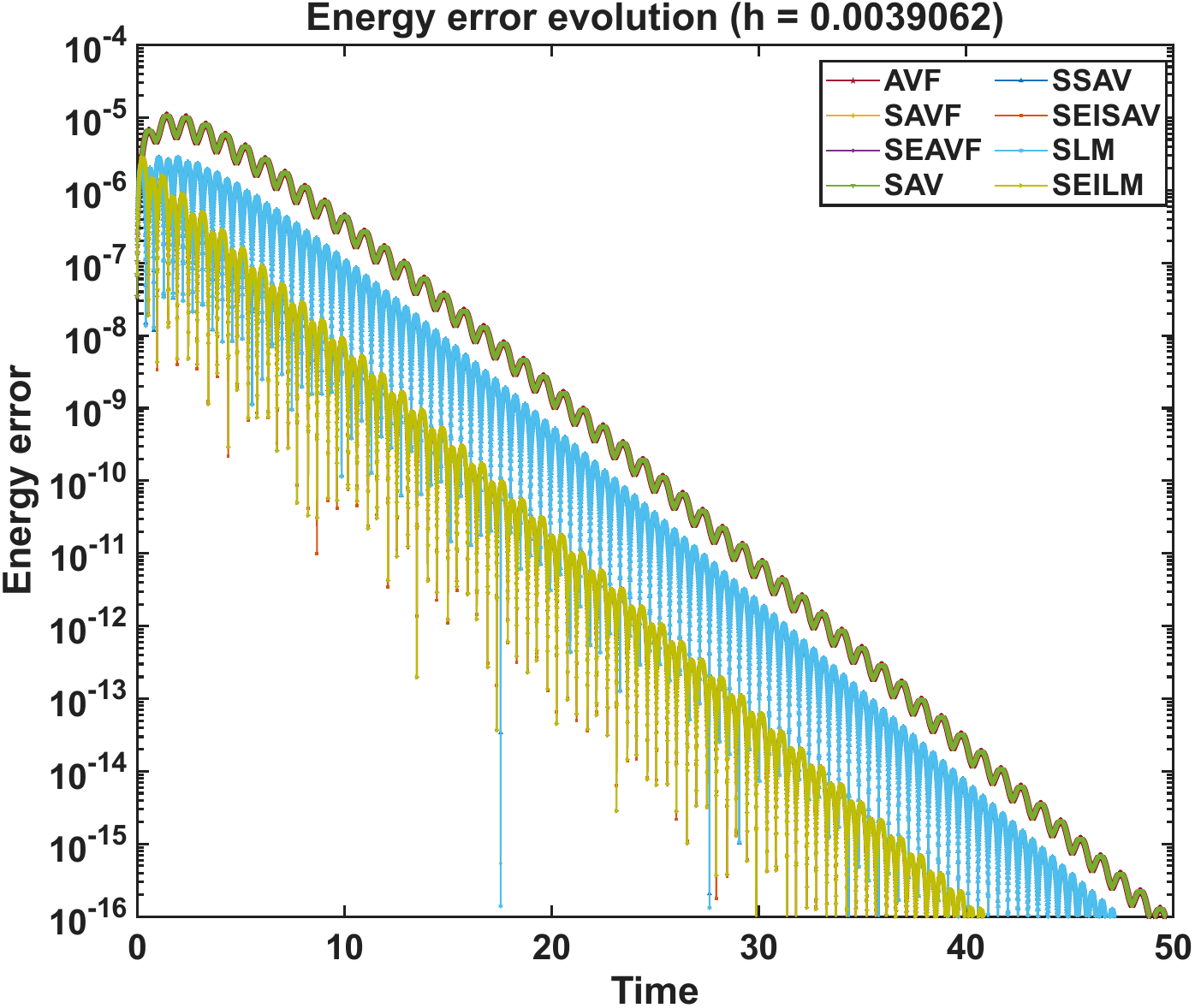}
  }
  \caption{Discrete energy evolution (left) and energy error (right)  for all methods applied to the FPU formulation in Example~2.1.}
  \label{fig:figure4FPU21}
\end{figure}

Figure~\ref{fig:orderB} demonstrates that all methods achieve a second-order convergence rate, and the exponential integrators are more accurate than the other schemes for each temporal step size. Figure~\ref{fig:efficiencyB} presents the CPU time versus accuracy and shows that the SAV/LM-based methods exhibit better computational efficiency, and the best two are the proposed  SEISAV and SEILM methods. Figure~\ref{fig:Energy decay B} depicts the evolution of the discrete energy \(H(z^n)\), confirming that the numerical solutions reproduce the expected monotone decay behavior. Figure~\ref{fig:Numerical solution B} reports the energy error with respect to the reference energy, and  we observe that SEILM, SEISAV and SEAVF yield comparable energy errors, as also observed in Example~1.

\paragraph{\textbf{Example 2.2 -- $\mathcal{S}$  is a negative semidefinite operator}}
 After spatial semi-discretization, system \eqref{eq:afpu-pde}  can be rewritten as a dissipative part perturbed by a linear damping term of the form \eqref{eq:model}, where the energy has the same form of \eqref{eq:afpu-diss-H}, and 
\[
S=\begin{pmatrix}0&I\\-I&-\gamma I\end{pmatrix},\qquad
D=\begin{pmatrix}0&0\\0&-\beta R\end{pmatrix}.
\]
\begin{figure}
  \centering
  \subfloat[Energy evolution plot\label{fig:orderB2}]{
    \includegraphics[width=0.46\textwidth]{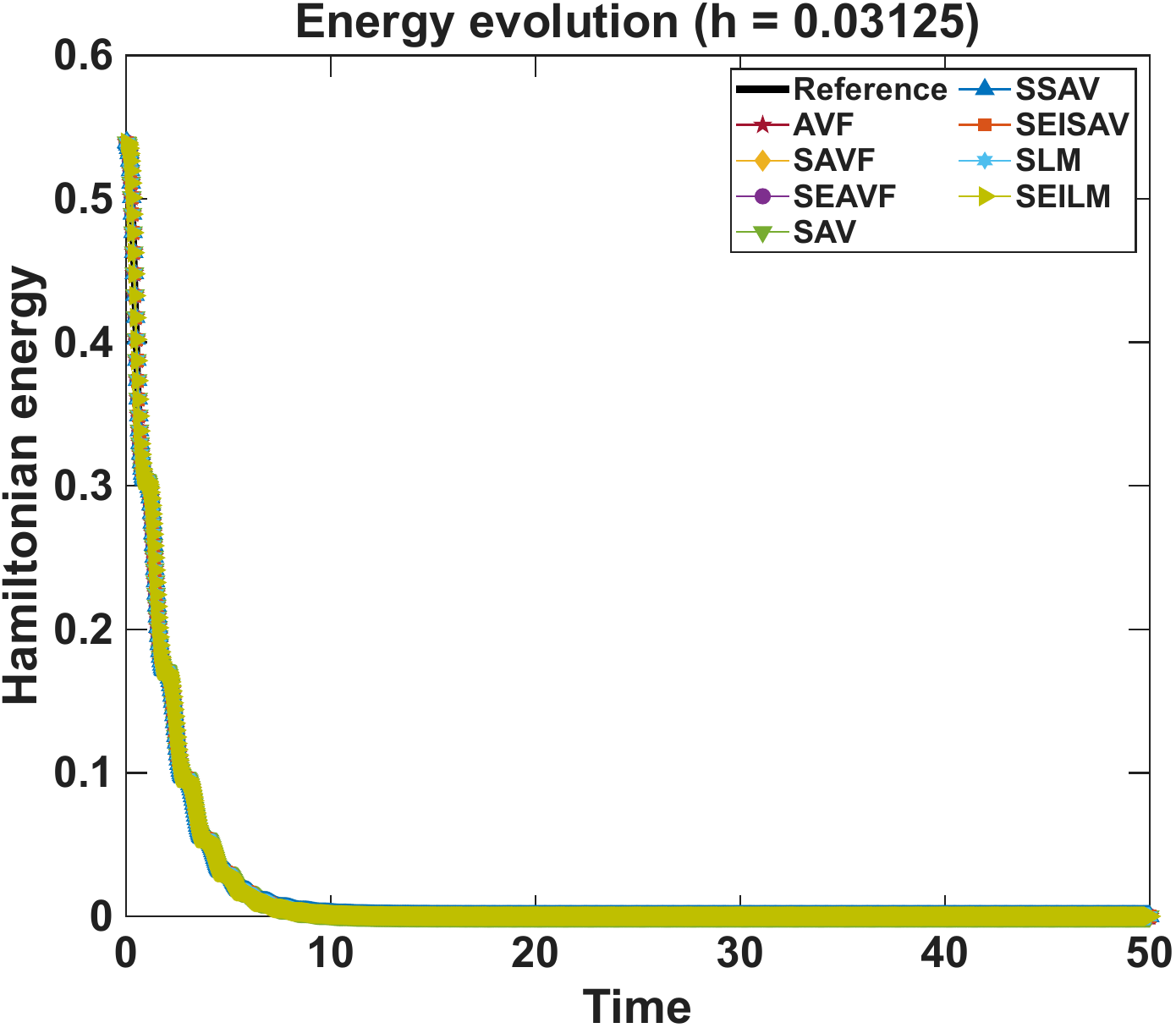}
  }\hfill
  \subfloat[Efficiency Comparison\label{fig:efficiencyB2}]{
    \includegraphics[width=0.46\textwidth]{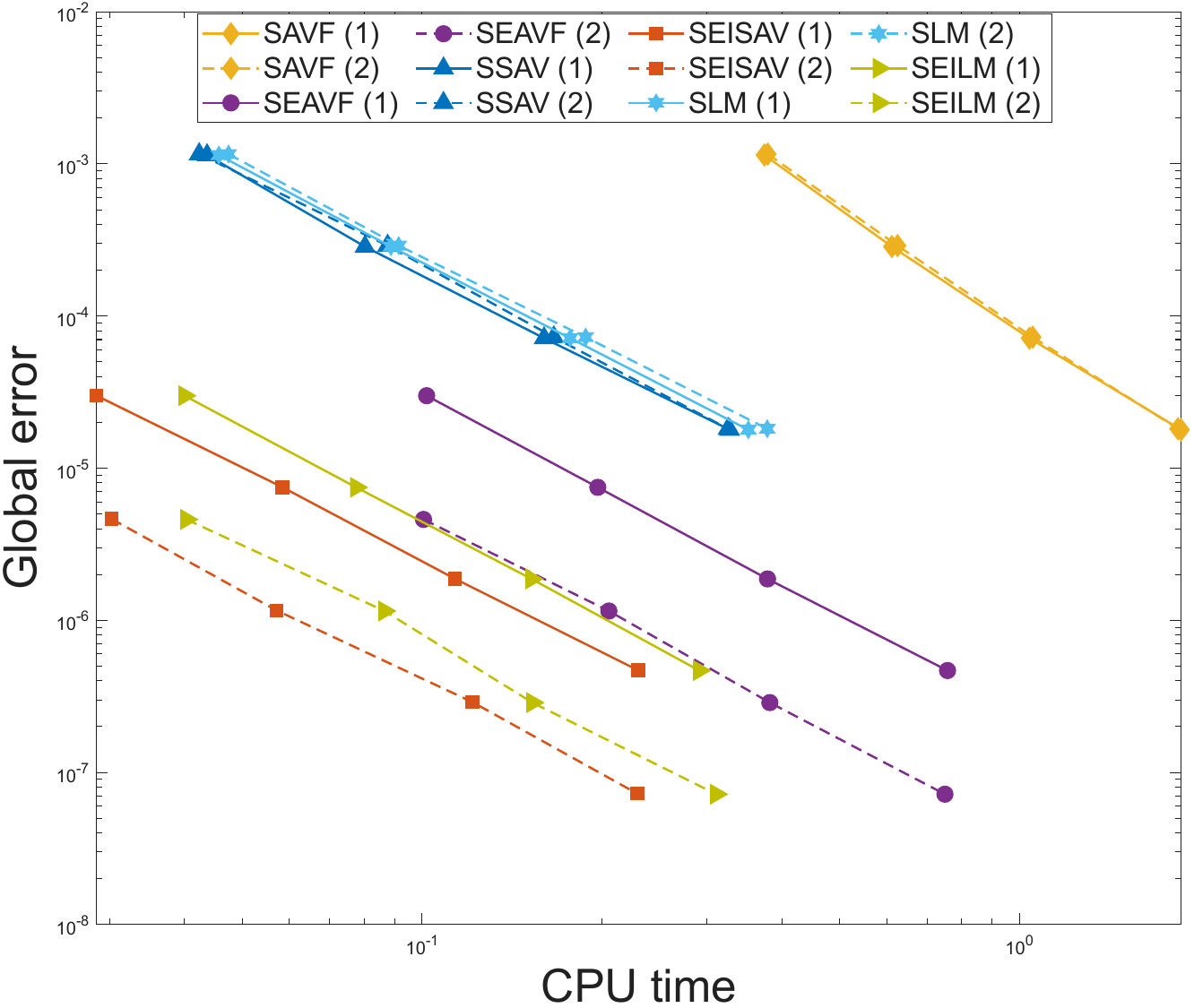}
  }

  \caption{Left: discrete energy evolution for all methods applied to the FPU formulation in Example~2.2; right: CPU time versus accuracy for all splitting schemes applied to the FPU formulations in Examples~2.1 and~2.2.}
  \label{fig:figure5FPU2122}
\end{figure}

We consider the same parameters and initial conditions as in Example 2.1. Figure~\ref{fig:orderB2} shows that all schemes reproduce the monotone energy decay of the exact solution. All methods also exhibit second-order convergence in time for this test; to avoid redundancy, we omit the corresponding error plot. Figure~\ref{fig:efficiencyB2} compares CPU time versus accuracy for the four splitting schemes applied to the FPU system under the two alternative operator formulations introduced in Examples~2.1 and~2.2, thereby summarizing both method-to-method performance and the impact of the chosen formulation. In both settings, SEISAV is the most efficient among the tested schemes, and the two proposed methods consistently achieve a better accuracy-cost trade-off than SEAVF. Moreover, for the exponential integrators, the second formulation (a dissipative subsystem perturbed by a linear damping term) yields smaller errors at comparable cost than the first formulation (a conservative subsystem perturbed by a linear damping term).

\paragraph{\textbf{Example 3:  Defocusing generalized KdV equation with linear damping}}

We consider the damped defocusing generalized Korteweg-de Vries (gKdV) equation
\cite{ShiXu2015DefocusingMKdV}
on the periodic interval \([0,L]\),
\begin{equation}\label{eq:gkdv}
    u_t+\varepsilon u_{xxx}-\partial_x\!\left(\sigma u^{k+1}\right)+\mu u=0,
    \quad x\in[0,L],\quad t>0,
\end{equation}
where \(\varepsilon>0\) is the dispersion coefficient, \(\sigma>0\) measures
the strength of the defocusing nonlinearity, \(k\in\mathbb N\) is the
nonlinearity exponent, and \(\mu>0\) is the linear damping coefficient.


Let \(N\in\mathbb N\) and introduce a uniform periodic grid with
\(\Delta x=L/N\). Denote \(u_j(t)\approx u(x_j,t)\) and set
\(z=(u_1,\ldots,u_N)^\top\in\mathbb R^N\). Under periodic boundary conditions,
we discretize \(\partial_x\) and \(\partial_{xx}\) by the standard central
difference operators, denoted by \(S\) and \(R\), respectively. This yields the
semi-discrete system of the form \eqref{eq:model}
with $S$ as defined in the previous line, \(D=\mu I\) and
\[
    H(z)=
    -\frac{\varepsilon}{2}z^\top Rz
    +
    \frac{\sigma}{k+2}\sum_{j=1}^N z_j^{k+2}.
\]

We set the model parameters as \(\sigma = 0.001\), \(\varepsilon = 1\) and \(k = 2\), and consider $L = 10$, $N = 20$, $\mu = 0.5$ and $T = 50$. The initial condition is set to be
\[
    u_j(0)=0.4\exp\left(-\frac{(x_j-L/2)^2}{2}\right),
    \qquad j=1,\ldots,N.   
\]
With $k=2$, the Hessian of the nonlinear potential is positive
semidefinite, while the periodic discrete Laplacian is negative
semidefinite. Thus, the energy-induced matrix $P_H(z)$ is positive
semidefinite. Since $D(t)=\mu I$ with $\mu>0$, the matrix condition
\eqref{eq:ECLD-matrix} is satisfied. Hence, the ECLD condition
\eqref{eq:ECLD} holds.

\begin{figure}[! htbp]
  \centering
  \subfloat[Order plot\label{fig:orderC}]{
    \includegraphics[width=0.46\textwidth]{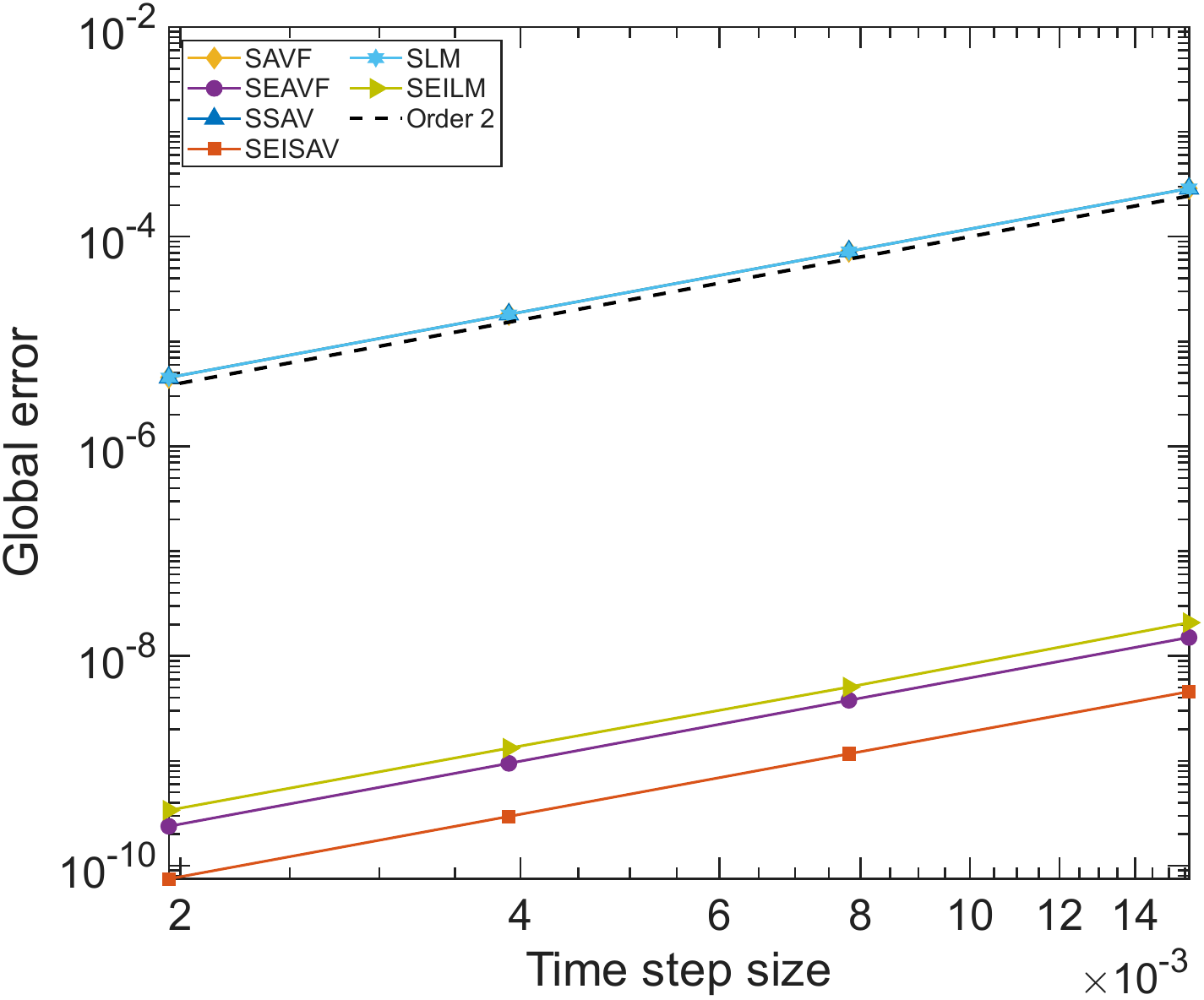}
  }\hfill
  \subfloat[Efficiency plot\label{fig:efficiencyC}]{
    \includegraphics[width=0.46\textwidth]{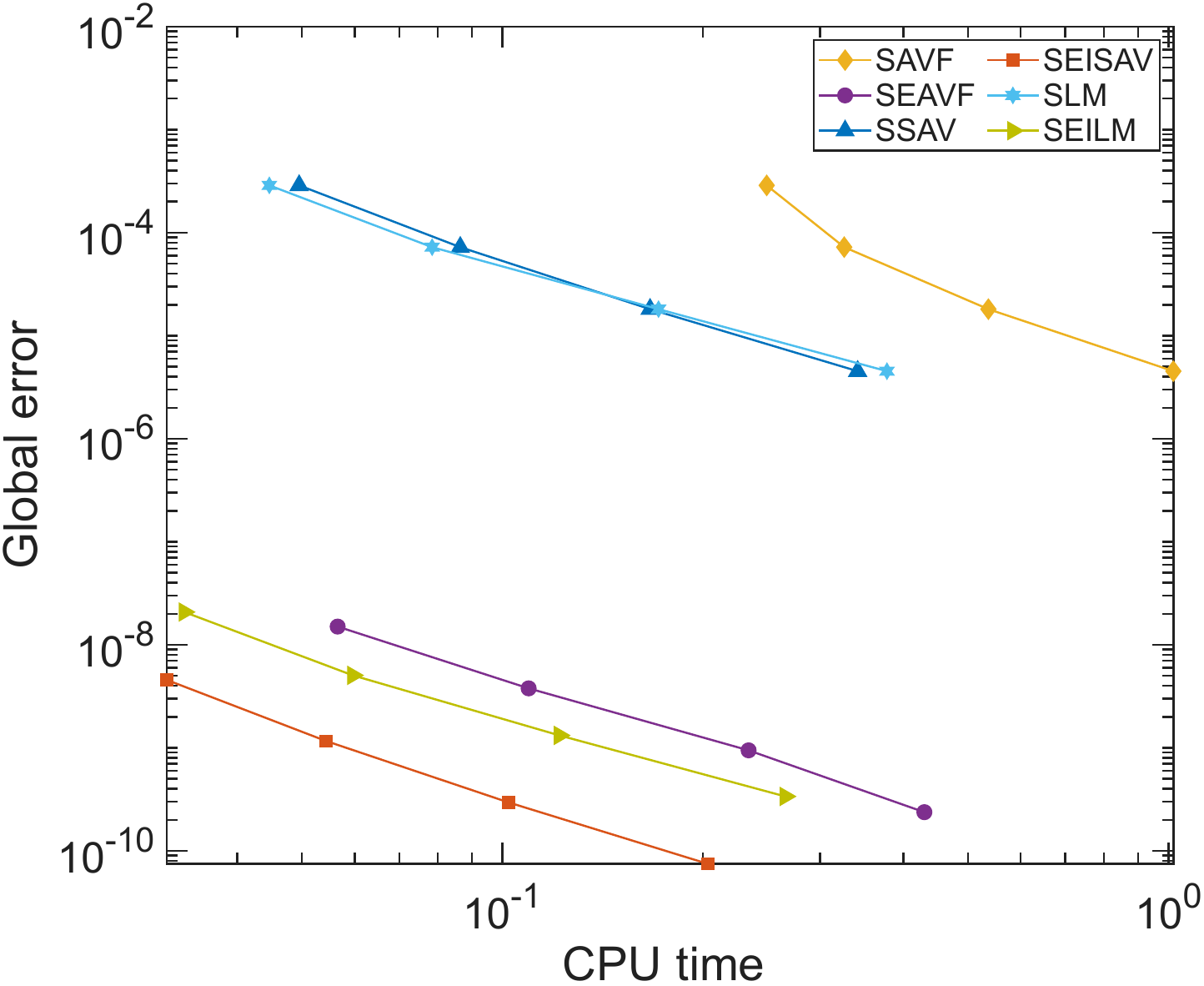}
  }
\caption{Comparison of convergence orders (left) and computational efficiency (right) for the damped gKdV equation.}
  \label{fig:figure6gKdV}
\end{figure}
Figure~\ref{fig:figure6gKdV}  compares the six structure-preserving schemes with different time step sizes \(h=2^{-(j+5)}\) \((j=1,2,3,4)\). 
Figure~\ref{fig:orderC}  shows that all methods achieve a second-order convergence rate. Figure~\ref{fig:efficiencyC} reports the CPU time versus accuracy, where we observe that SEISAV is substantially more efficient than the other schemes. SEILM is less efficient than SEISAV due to solving a nonlinear equation, but it is still more efficient than the remaining existing schemes. 
Figure~\ref{fig:Energy decay C} illustrates the evolution of the discrete energy \(H(z^n)\), showing that all methods reproduce monotone decay in this parameter regime. Figure~\ref{fig:Numerical solution C} shows that SEILM and SEAVF track the reference energy more closely than SEISAV, which is consistent with the role of the Lagrange-multiplier correction in the Hamiltonian substep. In this experiment, however, we find that SEILM is sensitive to parameter choices in \eqref{eq:gkdv}: the Newton iteration used to compute the scalar multiplier \(\eta\) in \eqref{eq:selmsav:core} is not always robust. For instance, when setting $\mu=0.1$ and $T=50$, numerical instabilities may arise. To address this issue, we employ the technique described in Remark \ref{rem:LM-solvability}.

\begin{figure}[H]
  \centering
  \subfloat[Energy decay plot\label{fig:Energy decay C}]{
    \includegraphics[width=0.46\textwidth]{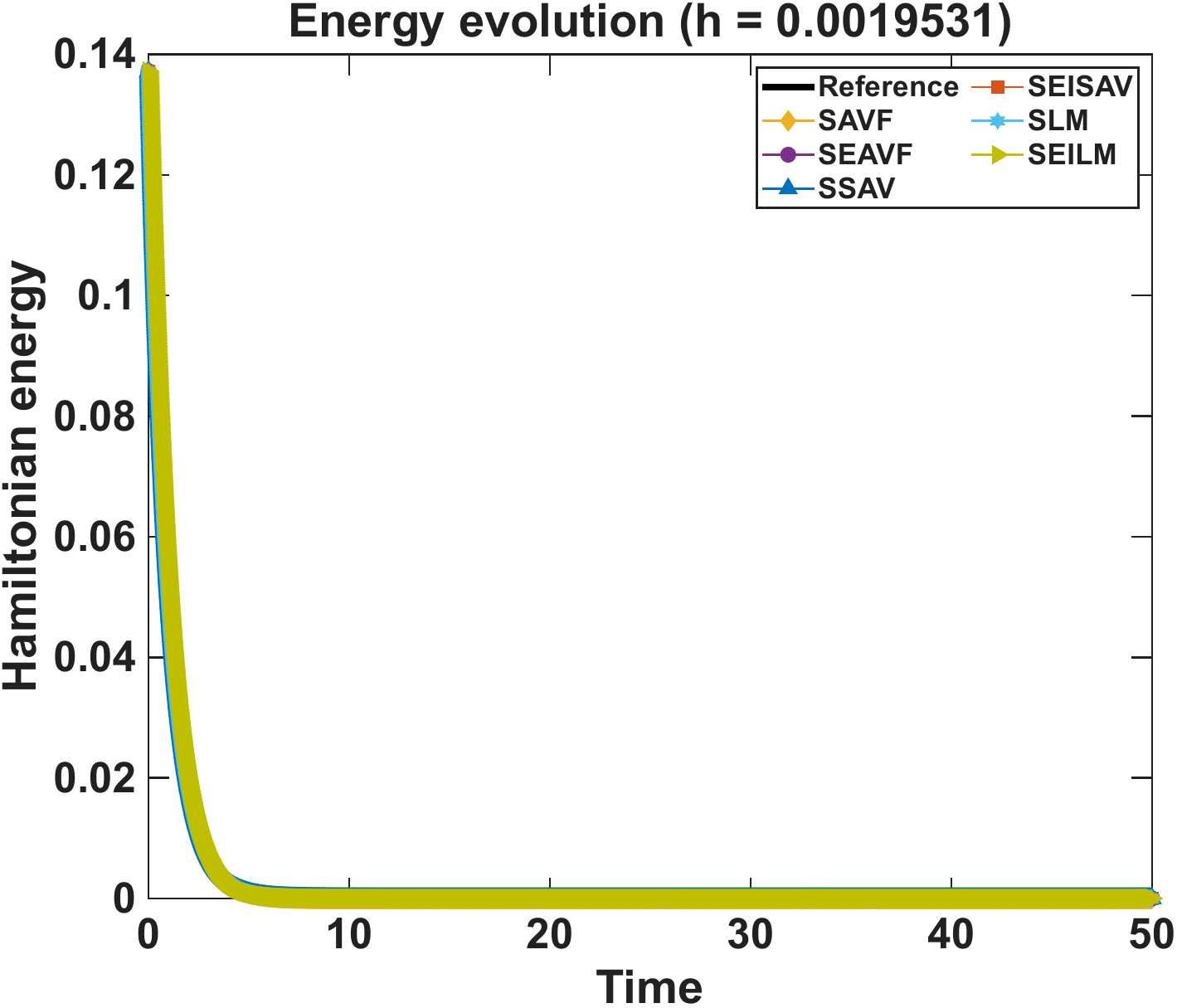}
  }\hfill
  \subfloat[Energy error plot\label{fig:Numerical solution C}]{
    \includegraphics[width=0.46\textwidth]{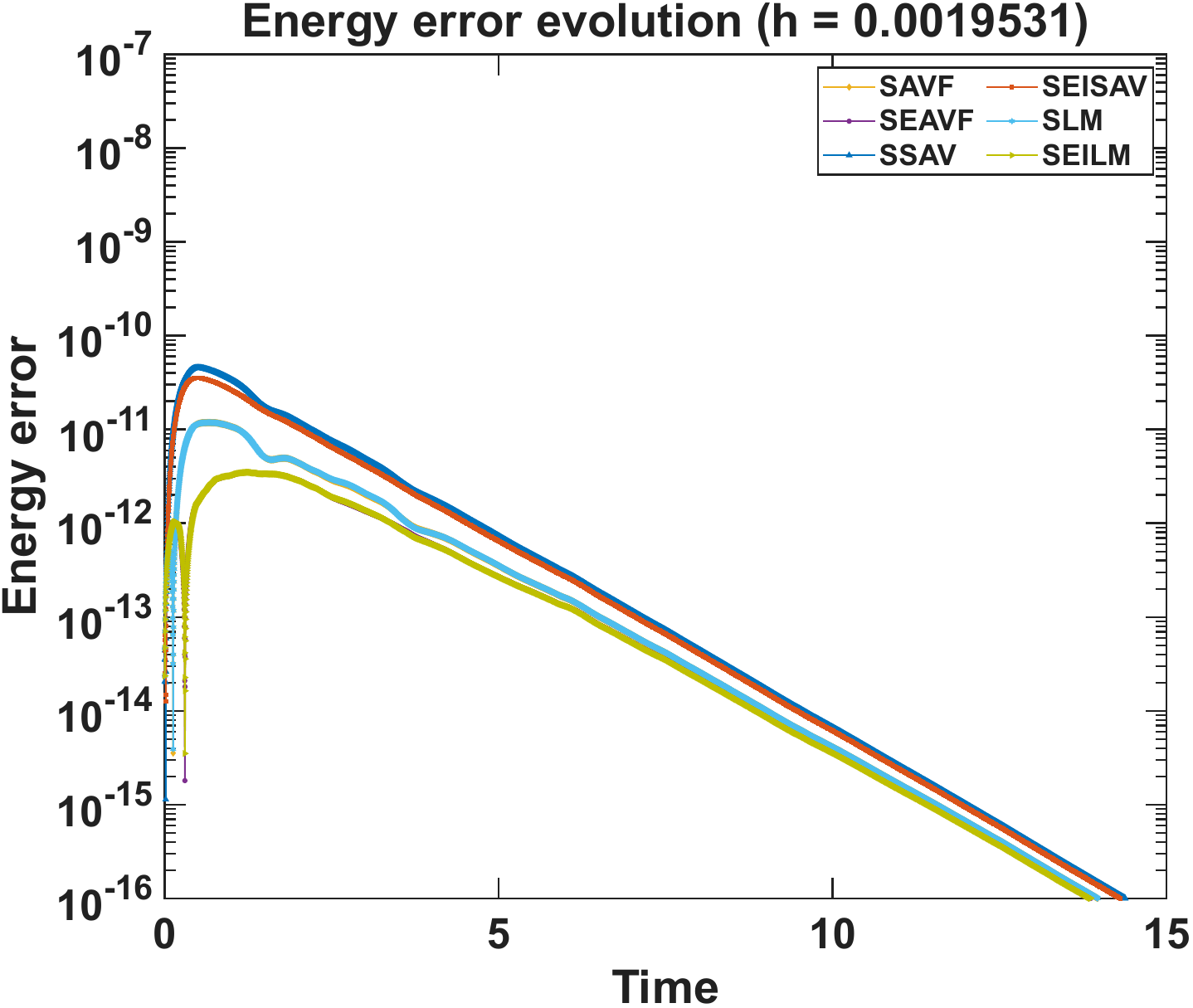}
  }

  \caption{Energy evolution (left) and energy error (right) for the damped gKdV equation.}
  \label{fig:figure7gKdV}
\end{figure}

\section{Conclusion}

We developed a unified framework for constructing efficient energy-stable splitting exponential integrators for damped Hamiltonian systems with linear perturbations, enabling reliable long-time simulation under dissipation.  Within this framework, we presented two splitting-based, dissipation-preserving integrators. The SEISAV scheme combines an exact update of the linear damping subflow with an exponential integrator based on SAV discretization for the Hamiltonian subflow, and the per-step computational cost reduces to solving a single scalar linear equation. The SEILM scheme further incorporates a Lagrange-multiplier formulation within the exponential integrator  so that the discrete dissipation is enforced for the original energy rather than a modified one, which only requires solving one nonlinear algebraic equation at each time step.

For SEISAV, we established the decay of a modified discrete energy under the
quadratic damping condition \(\operatorname{sym}(MD(t))\succeq0\), without
imposing an additional time-step restriction. For SEILM, assuming solvability of the scalar LM equation, we established original-energy decay under the ECLD condition, which provides a sufficient criterion for  full
Hamiltonian dissipativity of the damping subflow. Numerical experiments on the damped nonlinear
Klein-Gordon equation, the continuous FPU system, and the damped defocusing gKdV
equation demonstrate the expected convergence behavior, reliable energy
behavior, and competitive computational efficiency in the reported regimes.

A full convergence analysis of the proposed splitting schemes may build on
existing error estimates for SAV method~\cite{wang2021second}. However, the present setting
requires combining the SAV auxiliary-variable estimates with the Strang splitting
error and the local stability of the exponential SAV/LM Hamiltonian substeps. A
rigorous local solvability theory for the scalar LM equation further requires
nondegeneracy and small-step assumptions, consistent with recent LM analyses
~\cite{cheng2025unique}. These issues, together with nonlinear damping and
higher-order extensions, are left for future work. For nonlinear damping, the
splitting method may still be applicable, but the key issue is to design a
damping subsolver that retains the required full-energy dissipativity.
Higher-order extensions may draw on original-energy dissipative Runge-Kutta/exponential-Runge-Kutta ideas for
gradient flows~\cite{li2026global,zhang2024second,wang2025energy} and on existing
ETD/BDF-type energy-stable techniques
~\cite{cheng2019third,chen2020energy,hao2021third,cheng2022third}. However, for
high-order splitting schemes, preserving a rigorous energy law is nontrivial
because negative composition coefficients may break the monotonicity inherited
from the damping subflow.

\appendix
\section{Proof Details}\label{sec:esav review}
\begin{lemma}\label{lem:exp_congruence}
Let \(M\in\mathbb{R}^{d\times d}\) be symmetric and \(h> 0\). For
\[
B(h)=(e^{hSM})^{\top}Me^{hSM}-M,
\]
the following relations hold:
\[
B(h)=0 \quad \text{if }\quad S^{\top}=-S;\qquad
B(h)\preceq 0 \quad \text{if }\quad S\preceq 0.
\]
\end{lemma}

\begin{proof}
Let \(E(h)=e^{hSM}\) and \(N(h)=E(h)^{\top}ME(h)\). Then we have \(B(h)=N(h)-N(0)\) since \(N(0)=M\).
Differentiating \(N(h)\) gives
\begin{align*}
\frac{d}{dh}N(h)
&=E'(h)^{\top}ME(h)+E(h)^{\top}ME'(h) \\
&=E(h)^{\top}\Big((SM)^{\top}M+M(SM)\Big)E(h).
\end{align*}
Since \(M^{\top}=M\) and \((SM)^{\top}=M S^{\top}\), we can rewrite the above equation as 
\begin{equation}\label{eq:Nprime_S}
\frac{d}{dh}N(h)=E(h)^{\top}\,M(S^{\top}+S)M\,E(h).
\end{equation}
If \(S^{\top}=-S\), \eqref{eq:Nprime_S} implies \(\frac{d}{dh}N(h)=0\) for all \(h\ge 0\).
Hence \(N(h)\equiv N(0)=M\), and thus \(B(h)=0\).
If \(S\preceq 0\),  for any \(x\in\mathbb{R}^d\) we have 
\[
x^{\top}\frac{d}{dh}N(h)x
=(ME(h)x)^{\top}\,(S^{\top}+S)\,(ME(h)x)\le 0,
\]
implying \(\frac{d}{dh}N(h)\preceq 0\). Integrating from \(0\) to \(h\) yields
\[
N(h)-N(0)=\int_0^h \frac{d}{ds}N(s)\,ds \preceq 0,
\]
and therefore \(B(h)=N(h)-M\preceq 0\).
\end{proof}

\begin{lemma}
\label{app:EISAV-solvability}
\begin{revision}
Let $A=hSM$, $B_h=\varphi_1(A)S$, and $h>0$. Assume that $M=M^\top\succeq0$ and that either $S^\top=-S$ or $S\preceq0$. Then
\[
    B_h+B_h^\top\preceq0 .
\]
\end{revision}
\end{lemma}

\begin{proof}
\begin{revision}
Define
\[
    C(t)=t\varphi_1(tSM)S
    =\int_0^t e^{(t-\tau)SM}S\,d\tau .
\]
Then $C(0)=0$ and  $C'(t)=S+SMC(t)=S(I+MC(t))$. Let
$Q(t)=C(t)+C(t)^\top+C(t)^\top MC(t)$. A direct calculation gives
\[
    Q'(t)=(I+C(t)^\top M)(S+S^\top)(I+MC(t)).
\]
If $S^\top=-S$, then $Q'(t)=0$, and hence $Q(t)=0$. If $S\preceq0$, then $S+S^\top\preceq0$, so $Q'(t)\preceq0$ and hence $Q(t)\preceq0$. Since $C(h)=hB_h$, we obtain
\[
    h(B_h+B_h^\top)+h^2B_h^\top MB_h\preceq0 .
\]
Because $M\succeq0$, this implies $B_h+B_h^\top\preceq0$. 
\end{revision}
\end{proof}

\begin{proof}[Proof of Proposition 4]
Freeze the structure matrix at $\hat z^{\,n+\frac12}$, i.e., treat $\widetilde{S}(\hat z^{\,n+\frac12})$ as constant within the step.
Then, the CN discretization of the frozen extended system 
\begin{equation}\label{eq:CN-frozen-final}
\frac{\tilde z^{n+1}-\tilde z^{n}}{h}
=\widetilde{S}(\hat z^{\,n+\frac12})\,\nabla H\!\left(\frac{\tilde z^{n+1}+\tilde z^{n}}{2}\right)
\end{equation}
is a midpoint discrete-gradient scheme, where  $\tilde z^n=(z^n,V^n)$ is the numerical solution. With the consistent initialization $V^0=V(z^0)$, the LM constraint, i.e., the third line of the LM/CN scheme, implies
$V^{n+1}-V^n=V(z^{n+1})-V(z^n)$, and hence $V^n=V(z^n)$ for all $n$. Since $H(\tilde z)$ is quadratic in $\tilde z$, its  midpoint discrete gradient coincides with the AVF discrete gradient. Therefore, the LM/CN method can be viewed as the constrained discrete-gradient realization of \eqref{eq:CN-frozen-final} on the constraint manifold $V=V(z)$.
Consequently, the discrete-gradient identity yields
\begin{align*}
H(\tilde z^{n+1})-H(\tilde z^{n})
&=h\Big[\nabla H\!\left(\frac{\tilde z^{n+1}+\tilde z^{n}}{2}\right)\Big]^{\!\top}
\widetilde{S}(\hat z^{\,n+\frac12})
\nabla H\!\left(\frac{\tilde z^{n+1}+\tilde z^{n}}{2}\right)\\
&=h\,(\mu^{n+\frac12})^\top S\,\mu^{n+\frac12}.
\end{align*}
Therefore, the discrete energy is conserved when $S^\top=-S$ and is nonincreasing when $S\preceq0$. This completes the proof.
\end{proof}

\begin{proof}[Proof of Proposition 5]
Denoting by \(E=e^{A}\), we have $A\varphi_1(A)=E-I$. Setting $f^{n+\frac12}=\bar r^{\,n+\frac12} g^n$ and using \(\bar r^{\,n+\frac12}=\tfrac12(r^{n+1}+r^n)\), from the EISAV scheme, we can obtain
\begin{align}\label{eq:r_id}
(r^{n+1})^2-(r^n)^2=(f^{n+\frac12})^\top(z^{n+1}-z^n).
\end{align}
Assume first that \(M\) is invertible and set
\(\tilde f^{\,n+\frac12}=M^{-1}f^{n+\frac12}\).
Since \(A=hSM\), we have \(hSf^{n+\frac12}=A\tilde f^{\,n+\frac12}\). Furthermore, we have 
\[
h\varphi_1(A)S f^{n+\frac12}
=\varphi_1(A)\,A\,\tilde f^{\,n+\frac12}
=(E-I)\tilde f^{\,n+\frac12}.
\]
Consequently, from $z^{n+1}=Ez^n+(E-I)\tilde f^{\,n+\frac12}$, we deduce
\begin{align}\label{znext-reform}
z^{n+1}+\tilde f^{\,n+\frac12}=E\bigl(z^n+\tilde f^{\,n+\frac12}\bigr).
\end{align}
Defining \(w^n=z^n+\tilde f^{\,n+\frac12}\), we have \(z^{n+1}+\tilde f^{\,n+\frac12}=Ew^n\). Hence,
using \eqref{eq:r_id} and \(f^{n+\frac12}=M\tilde f^{\,n+\frac12}\), the discrete energy increment satisfies
\begin{align*}
&\widetilde{H}(z^{n+1},r^{n+1})-\widetilde{H}(z^{n},r^{n})\\
&= \frac12\Big((z^{n+1})^\top M z^{n+1} - (z^n)^\top M z^n \Big)
   + (z^{n+1}-z^n)^\top M \tilde f^{\,n+\frac12} \\
&= \frac12\Big((z^{n+1}+\tilde f^{\,n+\frac12})^\top M (z^{n+1}+\tilde f^{\,n+\frac12}) - (z^n+\tilde f^{\,n+\frac12})^\top M (z^n+\tilde f^{\,n+\frac12}) \Big) \\
&= \frac12 (w^n)^\top (E^\top M E - M) w^n.
\end{align*}

The last equality follows from equation \eqref{znext-reform}. Then, by Lemma~\ref{lem:exp_congruence},
we claim that  \(\widetilde{H}(z^{n+1},r^{n+1})=\widetilde{H}(z^{n},r^{n})\) if \(S\) is skew-symmetric, and
\(\widetilde{H}(z^{n+1},r^{n+1})\le \widetilde{H}(z^{n},r^{n})\) if \(S\preceq 0\).

If \(M\) is singular, we can replace \(M\) by
\(M_\varepsilon:=M+\varepsilon I\) with \(\varepsilon>0\), apply the above argument, and obtain the same conclusion by letting \(\varepsilon\to 0\).
\end{proof}

\begin{proof}[Proof of Proposition 6]
For exponential integrators, we have
\begin{equation*}\label{EISAVLMeq:phi1_id}
A\varphi_1(A)=E-I,
\end{equation*}
with \(E=e^{A}\). Denote by \(f^{n+\frac12}=\eta^{n+\frac12}\nabla V(\hat z^{\,n+\frac12})\) and
\(\tilde f^{\,n+\frac12}:=M^{-1}f^{n+\frac12}\) if \(M\) is invertible.
Then from EILM, we further have
\begin{equation*}
\begin{split}
z^{n+1}&=Ez^n+(E-I)\tilde f^{\,n+\frac12},\\
V(z^{n+1})-V(z^n)&=(z^{n+1}-z^n)^\top M\tilde f^{\,n+\frac12}.
\end{split}
\end{equation*}
Denoting by \(w^n=z^n+\tilde f^{\,n+\frac12}\), we have \(z^{n+1}+\tilde f^{\,n+\frac12}=Ew^n\).
Therefore, we have 
\begin{align*}
&H(z^{n+1})-H(z^n)\\
&=\frac12\Big((z^{n+1})^\top Mz^{n+1}-(z^n)^\top Mz^n\Big)
+\Big(V(z^{n+1})-V(z^n)\Big)\\
&=\frac12\Big((z^{n+1})^\top Mz^{n+1}-(z^n)^\top Mz^n\Big)
+(z^{n+1}-z^n)^\top M\tilde f^{\,n+\frac12}\\
&=\frac12\Big((z^{n+1}+\tilde f^{\,n+\frac12})^\top M(z^{n+1}+\tilde f^{\,n+\frac12})
-(z^{n}+\tilde f^{\,n+\frac12})^\top M(z^{n}+\tilde f^{\,n+\frac12})\Big)\\
&=\frac12\Big((Ew^{n})^\top MEw^{n}-(w^n)^\top Mw^n\Big)\\
&=\frac12 (w^n)^\top(E^\top M E - M)w^n.
\end{align*}
By Lemma~\ref{lem:exp_congruence}, \(E^\top M E-M=0\) if \(S^\top=-S\), and
\(E^\top M E-M\preceq 0\) if \(S\preceq 0\).

If \(M\) is singular, we can replace \(M\) by
\(M_\varepsilon:=M+\varepsilon I\) with \(\varepsilon>0\), apply the above argument, and obtain the same conclusion by letting \(\varepsilon\to 0\).
\end{proof}

\end{document}